\documentclass[10pt]{amsart}
\usepackage{latexsym}
\usepackage{amssymb}
\usepackage{amsmath,amstext,amscd}
\usepackage{mathrsfs}
\usepackage{graphicx}
\usepackage{verbatim}
\usepackage[all]{xy}
\usepackage{amsfonts}
\usepackage{indentfirst}
\usepackage{enumerate}
\usepackage{txfonts}
\usepackage{pxfonts}
\usepackage{float}
\usepackage{subcaption}
\usepackage[linkcolor=black]{hyperref}
\usepackage{bezier,longtable}
\usepackage[labelformat=empty]{caption}

\newtheorem{theo}{Theorem}[section]

\newtheorem{lemma}[theo]{Lemma}

\newtheorem{corollary}[theo]{Corollary}
\newtheorem{prop}[theo]{Proposition}

\newtheorem{conjecture}[theo]{Conjecture}
\renewenvironment{proof}{ \emph{Proof}}{$\Box$}
\newtheorem{defi}{Definition}[section]

\newcommand {\ZZ} {\mathbb {Z}}

\newcommand {\CC} {\mathbb {C}}
\newcommand{\TT}{\mathbb{T}}

\newcommand{\wt}{\mathrm{wt}}
\newcommand{\jj}{\mathfrak{j}}
\newcommand{\so}{\mathfrak{so}}
\renewcommand{\sp}{\mathfrak{sp}}
\renewcommand{\sl}{\mathfrak{sl}}

\newcommand{\oj}{\mathfrak{g}}
\newcommand{\ojl}{\mathfrak{gl}}

\newcommand{\kk}{\mathfrak{k}}

\newcommand{\id}{\mathrm{Id}}

\newcommand{\oo}{\mathfrak{o}}

\renewcommand{\aa}{\mathfrak{a}}
\newcommand{\hh}{\mathfrak{h}}

\renewcommand{\gg}{\mathfrak{g}}

\renewcommand{\ll}{\mathfrak{l}}

\renewcommand {\phi} {\varphi}

\newcommand{\tr}{\mathrm{tr}}

\newcommand{\Int}{\mathrm{Int}}
\newcommand{\alg}{\mathrm{alg}}
\newcommand{\fin}{\mathrm{fin}}

\renewcommand{\span}{\mathrm{span}}
\newcommand{\gl}{\mathfrak{gl}}

\newcommand{\Soc}{\mathrm{soc}}

\newcommand{\Ker}{\mathrm{ker}}
\newcommand{\Hom}{\mathrm{Hom}}
\newcommand{\End}{\mathrm{End}}

\def\cplus{\hbox{$\subset${\raise0.3ex\hbox{\kern -0.55em ${\scriptscriptstyle +}$}}\ }}

\def\clplus{\hbox{$\subset${\raise0.3ex\hbox{\kern -0.55em ${\scriptscriptstyle +}$}}\ }}
\def\crplus{\hbox{$\supset${\raise1.05pt\hbox{\kern -0.55em ${\scriptscriptstyle +}$}}\ }}

\title{Tensor Representations of Mackey Lie algebras and their dense subalgebras}
\author[Ivan Penkov]{\;Ivan~Penkov}
\address{
Ivan Penkov
\newline School of Engineering and Sciences, Mathematics 
\newline Jacobs University Bremen
\newline Campus Ring 1
\newline Bremen, Bremen 28759, Germany}
\email{i.penkov@jacobs-university.de}

\author[Vera Serganova]{\;Vera Serganova}
\address{
Vera Serganova
\newline Department of Mathematics
\newline University of California, Berkeley
\newline Berkeley, CA 94720-3840, USA}
\email{serganov@math.berkeley.edu}

\begin{document}

\maketitle

\begin{abstract}
 In this article we review the main results of the earlier papers \cite{pstyr}, \cite{ps} and \cite{DPS}, and establish related new results in considerably greater generality. We introduce a class of infinite-dimensional Lie algebras $\gg^{M}$, which we call Mackey Lie algebras, and define monoidal categories $\TT_{\gg^M}$ of tensor $\gg^M-$modules. We also consider dense subalgebras $\aa \subset \gg^M$ and corresponding categories $\TT_\aa$. The locally finite Lie algebras $\sl(V,W), \oo(V), \sp(V)$ are dense subalgebras of respective Mackey Lie algebras. Our main result is that if $\gg^M$ is a Mackey Lie algebra and $\aa \subset \gg^M$ is a dense subalgebra, then the monoidal category $\TT_\aa$ is equivalent to $\TT_{\sl(\infty)}$ or $\TT_{\oo(\infty)}$; the latter monoidal categories have been studied in detail in \cite{DPS}. A possible choice of $\aa$ is the well-known Lie algebra of generalized Jacobi matrices.
 
 
\end{abstract}

\section*{Introduction}\label{sect:intro}
This paper combines a review of some results on locally finite Lie
algebras, mostly from \cite{pstyr}, \cite{ps} and \cite{DPS}, with new
results about categories of representations of a class of (non-locally
finite) infinite-dimensional Lie algebras which we call Mackey Lie
algebras. Locally finite Lie algebras (i.e. Lie algebras in which any
finite set of elements generates a finite-dimensional Lie subalgebra)
and their representations have been gaining the attention of
researchers in the past 20 years. An incomplete list of references in this topic is: \cite{BA1}, \cite{BB}, \cite{BS}, \cite{dip}, \cite{dip3}, \cite{DPS}, \cite{DPSn}, \cite{DPW}, \cite{DaPW} \cite{N}, \cite{Na}, \cite{NP}, \cite{NS}, \cite{O}, \cite{ps}, \cite{pstyr}, \cite{PZ}. In particular, in \cite{pstyr}, \cite{ps} and \cite{DPS} integrable representations of the three
classical locally finite Lie algebras 
$\gg = \sl(\infty), \oo(\infty), \sp(\infty)$ 
have been studied from various points of view. An important step in
the development of the representation theory of 
these Lie algebras has been the introduction of the category of tensor modules $\TT_\gg$ in \cite{DPS}.

In the present article we shift the focus to understanding a natural
generality in which the category $\TT_\gg$ is defined. 
In particular, we consider the finitary locally simple Lie algebras
$\gg = \sl(V,W), \oo(V), \sp(V)$, 
where $V$ is an arbitrary vector space (not necessarily of countable dimension), and either a non-degenerate pairing 
$V \times W \rightarrow \CC$ is given, or $V$ is equipped with a
non-degenerate symmetric, or antisymmetric, form. In sections 1-5 we
reproduce the most important results from \cite{pstyr} and \cite{DPS}
in this greater generality. In fact, we study five different
categories of integrable modules, see subsection \ref{sect:3.6}, but
pay maximum attention to the category $\TT_\gg$. The central new
result in this part of the paper is Theorem \ref{equivmain}, claiming
that the category $\TT_\gg$ for $\gg = \sl(V,W), \oo(V), \sp(V)$ is canonically equivalent, as a
monoidal category, to the respective category 
$\TT_{\sl(\infty)}, \TT_{\oo(\infty)}$ or $\TT_{\sp(\infty)}$. It is
shown in \cite{DPS} 
that each of the latter categories is Koszul and that
$\TT_{\sl(\infty)}$ is self-dual Koszul, while 
$\TT_{\oo(\infty)}$ and $\TT_{\sp(\infty)}$ are not self-dual but are equivalent.

In the second part of the paper, starting with section \ref{sect:6},
we explore several new ideas. The first one is that, 
given a non-degenerate pairing $V \times W \rightarrow \CC$ between
two vector spaces, or a non-degenerate symmetric or antisymmetric form
on a vector space $V$, there is a canonical, in general non-locally
finite, Lie algebra attached to this datum. 
Indeed, fix a pairing $V \times W \rightarrow \CC$. Then the Mackey
Lie algebra $\gl^M(V,W)$ is the Lie algebra of all endomorphisms of
$V$ whose duals keep $W$ stable (this definition is given in a more
precise form at the beginning of section 
\ref{sect:6}). Similarly, if $V$ is equipped with a non-degenerate
form, the respective Lie algebra 
$\oo^M(V)$ or $\sp^M(V)$ is the Lie algebra of all endomorphisms of $V$ for which the form is invariant. 

The Lie algebras $\gl^M(V,W),\oo^M(V),\sp^M(V)$ are not simple as they
have obvious ideals: these are respectively 
$\gl(V,W) \oplus \CC\mathrm{Id}$, $\oo(\infty)$, and $\sp(\infty)$. However, we
prove that, if both $V$ and $W$ are countable dimensional, the
quotients $\gl^M(V,W)/(\gl(V,W)\oplus\CC\mathrm{Id}), \oo^M(V)/\oo(V),
\sp^M(V)/\sp(V)$ are simple Lie algebras. 
This result is an algebraic analogue of the simplicity of the Calkin algebra in functional analysis.

Despite the fact that the Lie algebras $\gl^M(V,W), \oo^M(V),
\sp^M(V)$ are completely natural objects, the representation theory of
these Lie algebras has not yet been explored. We are undertaking the
first step of such an exploration by introducing the categories of
tensor modules $\TT_{\gg^M}$ for $\gg^M = \gl^M(V,W),
\oo^M(V),\sp^M(V)$. Our main result about these categories is 
Theorem \ref{mainres_sect7}
which implies that $\TT_{\gl^M(V,W)}$ is equivalent to
$\TT_{\sl(\infty)}$, and $\TT_{\oo^M(V)}$ and 
$\TT_{\sp^M(V)}$ are equivalent respectively to $\TT_{\oo(\infty)}$ and $\TT_{\sp(\infty)}$.

A further idea is to  consider dense subalgebras $\aa$ of the Lie algebras $\gg^M$ (see the definition in section \ref{sect:7}). We show that if $\aa \subset \gg$ is a dense subalgebra, the category $\TT_\aa$, whose objects are tensor modules of $\gg$ considered as $\aa-$modules, is canonically equivalent to $\TT_{\gg^M}$, and hence to one of the categories $\TT_{\sl(\infty)}$ or $\TT_{\oo(\infty)}$. It is interesting that this result applies to the Lie algebra of generalized Jacobi matrices (infinite matrices with ``finitely many non-zero diagonals'') which has been studied for over 30 years, see for instance \cite{FT}.

In short, the main point of the present paper is that the categories of tensor modules $\TT_{\sl(\infty)}, \TT_{\oo(\infty)}, \TT_{\sp(\infty)}$ introduced in \cite{DPS} are in some sense universal, being naturally equivalent to the respective categories of tensor representations of a large class of, possibly non-locally finite, infinite-dimensional Lie algebras.

\section*{Acknowledgements}\label{sect:acknow.}
Both authors thank the Max Planck Institute for Mathematics in Bonn where a preliminary draft of this paper was written in 2012. We also acknowledge support by the DFG Priority Program $1388$ ``Representation theory''. Vera Serganova acknowledges support from the NSF via grant 1303301 and thanks Jacobs University for its hospitality.

\section{Preliminaries}\label{sect:1}

The ground field is $\CC$. By $M^*$ we denote the dual space of a vector space $M$, i.e. $M^* = \mathrm{Hom}_{\CC}(M,\CC)$. $S_n$ stands for the symmetric group on $n$ letters. The sign $\subset$ denotes not necessarily strict inclusion. Under a \textit{natural representaion} (or a \textit{natural module}) of a classical simple finite-dimensional Lie algebra we understand a simple non-trivial finite-dimensional representation of minimal dimension.

In this paper $\gg$ denotes a \textit{locally simple locally finite} Lie algebra, i.e. an infinite-dimensional Lie algebra $\gg$ 
obtained as the direct limit $\varinjlim\,\gg_\alpha$ of a directed system of embeddings (i.e. injective homomorphisms) $\gg_\alpha \hookrightarrow\gg_\beta$
of finite-dimensional simple Lie algebras parametrized by a directed set of indices. It is clear that any such $\gg$ is a simple Lie algebra.
If $\gg$ is countable dimensional, then the above directed set can always be chosen as $\ZZ_{\geq 1}$, and the corresponding directed system can be chosen as a chain 
\begin{equation}\label{def_gg}
 \gg_1 \hookrightarrow \gg_2 \hookrightarrow \dots \gg_i \hookrightarrow \gg_{i+1} \hookrightarrow \dots.
\end{equation}
In this case we write
$\gg=\varinjlim \gg_i$. Moreover, if $\gg_i = \sl(i+1)$, then up to isomorphism there is only one such Lie algebra which we denote by $\sl(\infty)$. Similarly, if $\gg_i = \oo(i)$ or $\gg_i = \sp(2i),$ up to isomorphism one obtains only two Lie algebras: $\oo(\infty)$ and $\sp(\infty)$. The Lie algebras $\sl(\infty)$,  $\oo(\infty)$,  $\sp(\infty)$ are often referred to as the \textit{finitary locally simple Lie algebras} \cite{BA1}, \cite{BA2}, \cite{BS}, or as the \textit{classical locally simple Lie algebras} \cite{ps}. 

A more general (and very interesting) class of locally finite locally simple  Lie algebras are the diagonal locally finite Lie algebras introduced by Y. Bahturin and H. Strade in \cite{BHS}.
We recall that an injective homomorphism $\gg_1 \hookrightarrow \gg_2$ of simple classical Lie algebras of the same type $\sl, \oo, \sp,$ is \textit{diagonal} if the pull-back $V_{\gg_2\downarrow\gg_1}$ of a natural representation $V_{\gg_2}$ of $\gg_2$ to $\gg_1$ is isomorphic to a direct sum of copies of a natural representation $V_{\gg_1},$ of its dual $V_{\gg_1}^*,$ and of the trivial $1-$dimensional representation. In this paper, under a \textit{diagonal Lie algebra $\gg$} we mean an infinite-dimensional Lie 
algebra obtained as the limit of a directed system of diagonal homomorphisms of classical simple Lie algebras $\gg_\alpha$. We say that a diagonal Lie algebra \textit{is of type} $\sl$ (resprectively, $\oo$ or $\sp$) if all $\gg_\alpha$ can be chosen to have type $\sl$ (respectively, $\oo$ or $\sp$). 

Countable-dimensional diagonal Lie algebras have been classified up to
isomorphism by A. Baranov and A. Zhilinskii \cite{BZ}. S. Markouski
\cite{M} has determined when there is an embedding $\gg
\hookrightarrow \gg'$ for given countable-dimensional diagonal Lie
algebras $\gg$ and $\gg'$. If both $\gg$ and $\gg'$ are classical
locally simple Lie algebras, then an embedding 
$\gg \hookrightarrow \gg'$ always exists, and such embeddings have been studied in detail in \cite{dip2}.

Let $V$ and $W$ be two infinite-dimensional vector spaces with a non-degenerate pairing $V \times W \rightarrow \CC$. 
G. Mackey calls such a pair $V,\,W$ a \textit{linear system} and was the first to study linear systems in depth \cite{m}.
The tensor product $V\otimes W$ is an associative algebra (without
identity), and we denote the corresponding Lie algebra by 
$\gl(V,W)$.
The pairing $V\times W\to\CC$ induces a homomorphism of Lie algebras
$\mathrm{tr} : \gl(V,W)\to\CC$. The kernel of this homomorphism is
denoted by $\sl(V,W)$. The Lie algebra $\sl(V,W)$ is a locally simple
locally finite Lie algebra. A corresponding directed system is given
by $\{\sl(V_f,W_f)\}$, where $V_f$ and $W_f$ run over all
finite-dimensional subspaces $V_f\subset V, W_f\subset W$ such that
the restriction of the pairing $V \times W \rightarrow \CC$ to
$V_f\times W_f$ is non-degenerate. If $V$ and $W$ are countable
dimensional, then $\sl(V,W)$ is isomorphic to $\sl(\infty)$. In what
follows we call a pair of finite-dimensional subspaces $V_{f} \subset
V$, $W_{f} \subset W$ a \textit{finite-dimensional non-degenerate
  pair} if the restriction of the pairing $V \times W \rightarrow \CC$
to $V_{f} \times W_{f}$ is non-degenerate. We can also define 
$\gl(V,W)$ as a Lie algebra of finite rank linear operators in
$V\oplus W$ preserving $V,W$ and the pairing $V\times W\to\mathbb C$.

If $V$ and $W$ are countable dimensional then, up to isomorphism, there is only one linear system \cite{m}. In this case we set $V_*:=W$. 
According to Mackey \cite{m}, there exists a basis $\lbrace
v_1,v_2,\dots\rbrace$ of $V$ such that 
$V_* = \span \lbrace v^*_1, v^*_2, \dots\rbrace$, 
where $\lbrace v^*_1, v^*_2, \dots\rbrace$ is the set of linear
functionals dual to $\lbrace v_1,v_2,\dots\rbrace$, i.e. $v_i^*(v_j) =
\delta_{ij}$. The choice of such basis identifies $\gl(V,V_*)$ with
the Lie algebra $\gl(\infty)$ consisting of infinite matrices
$X=(x_{ij})_{i\geq 1,j\geq 1}$ with finitely many non-zero
entries. The Lie algebra $\sl(V,W)$ is identified with 
$\sl(\infty)$ realized as the Lie algebra of traceless matrices 
$X = (x_{ij})_{i \geq 1, j\geq 1}$ with finitely many non-zero entries. 

Now  let $V$ be a vector space endowed with a non-degenerate symmetric (respectively, antisymmetric) form $(\cdot,\cdot)$. 
Then $\Lambda^2V$ (respectively, $S^2V$) has a Lie algebra structure, defined by 
$$[v_1\wedge v_2,w_1\wedge w_2]=-(v_1,w_1)v_2\wedge w_2+(v_2,w_1)v_1\wedge w_2 +(v_1,w_2)v_2\wedge w_1-(v_2,w_2)v_1\wedge w_1$$ 
(respectively, by $$[v_1 v_2,w_1 w_2]=(v_1,w_1)v_2 w_2+(v_2,w_1)v_1 w_2 +(v_1,w_2)v_2 w_1+(v_2,w_2)v_1 w_1.$$ 
We denote the Lie algebra  $\Lambda^2V$ by $\oo(V)$, and the Lie algebra $S^2V$ by $\sp(V)$. Let $V_f\subset V$ be an $n$-dimensional subspace such that the 
restriction of the form on $V_f$ is non-degenerate. Then $\oo(V_f)\subset \oo(V)$ (respectively, $\sp(V_f)\subset \sp(V)$) is a simple subalgebra 
isomorphic to $\oo(n)$ (respectively, $\sp(n)$). Therefore, $\oo(V)$ (respectively, $\sp(V)$) is the direct limit of all its subalgebras $\oo(V_f)$ 
(respectively, $\sp(V_f)$). This shows that both $\oo(V)$ and $\sp(V)$ are locally simple locally 
finite Lie algebras. We can also identify $\oo(V)$ (respectively. $\sp(V)$ with the Lie subalgebra of finite rank operators in $V$ preserving the form  
$(\cdot,\cdot)$.

If $V$ is countable dimensional, there always is a basis $\lbrace v_i,w_j\rbrace_{i,j \in \ZZ}$ of $V$ such that $\span\lbrace v_i\rbrace_{i \in \ZZ}$ and $\span\lbrace w_j\rbrace_{j \in \ZZ}$ are isotropic spaces and $(v_i, w_j) = 0$ for $i \neq j$, $(v_i, w_i)=1$. Therefore, in this case $\oo(V) \simeq \oo(\infty)$ and  $\sp(V)\simeq \sp(\infty)$.

Note that, if $V$ is not finite or countable dimensional, then $V$ may have several inequivalent 
non-degenerate symmetric forms. Indeed, let for instance $V:=W\oplus W^*$ for some countable-dimensional space $W$. Extend the pairing between $W$ and $W^*$ to a 
non-degenerate symmetric form $(\cdot,\cdot)$ on $V$ for which $W$ and $W^*$ are both isotropic. It is clear that $W$ is a maximal isotropic subspace of $V$. 
On the other hand, choose a basis $\bf b$ in $V$ and let $(\cdot,\cdot)'$ be the symmetric form on $V$ for which $\bf b$ is an orthonormal basis. Then 
$V$ does not have countable-dimensional maximal isotropic subspaces for the form $(\cdot,\cdot)'$. Hence the forms $(\cdot,\cdot)$ and $(\cdot,\cdot)'$  are 
not equivalent.

\begin{prop}\label{prop:1.29.9}
a) Two Lie algebras $\sl(V,W)$ and $\sl(V',W')$ are isomorphic if and only if the linear systems $V\times W\rightarrow \CC$ and 
$V' \times W' \rightarrow \CC$ are isomorphic.

b) Two Lie algebras $\oo(V)$ and $\oo(V')$ (respectively, $\sp(V)$ and $\sp(V')$) are isomorphic if and only if there is an isomorphism of vector spaces 
$V\simeq V'$ transferring the form defining $\oo(V)$ (respectively $\sp(V)$) into the form defining $\oo(V')$ (respectively, $\sp(V')$).

\end{prop}

We first prove a lemma.

\begin{lemma}\label{lemma:1.4.10}(cf. Proposition $2.3$ in \cite{dip2})

a) Let $\gg_1 \subset \gg_3$ be an inclusion of classical finite-dimensional simple Lie algebras such that a natural $\gg_3$-module restricts to $\gg_1$ as 
the direct sum of a natural $\gg_1$-module and a trivial $\gg_1$-module. If $\gg_2$ is an intermediate classical simple subalgebra, 
$\gg_1\subseteq\gg_2\subseteq\gg_3$, 
then a natural $\gg_3$-module restricts to $\gg_2$ as the direct sum of a natural $\gg_2$-module plus a trivial module. 

b) Assume $\mathrm{rk} \gg_1 > 4$.
If $\gg_1\simeq \sl(i)$, then $\gg_2$ is isomorphic to $\sl(k)$ for some $k\geq i$. If $\gg_3\simeq \oo(j)$   
(respectively, $\sp(2j)$), then $\gg_2$ is isomorphic to $\oo(k)$ (respectively, $\sp(2k)$) for some $k\leq j$.
\end{lemma}

\begin{proof} Let $V_3$ be a natural $\gg_1-$module. We have a
  decomposition of $\gg_2-$modules, $V_3=V_1\oplus W$, where 
$V_1$ is a natural 
$\gg_1-$module and $W$ is a trivial $\gg_1$-module. Let $V'\subset
  V_3$ be the minimal $\gg_2$-submodule containing $V_1$. Then 
$V_3=V'\oplus W'$, where $W'$ 
is a complementary $\gg_2$-submodule. Since $\gg_1$ acts trivially on $W'$ and $\gg_2$ is simple, we obtain that $W'$ is a trivial $\gg_2$-module and $V'$ is a 
simple $\gg_2$-module. 

We now prove that $V'$ is a natural $\gg_2$-module. Recall that for an arbitrary non-trivial
module $M$ over a simple Lie algebra $\kk$ the symmetric form $B_M(X,Y)=\mathrm{tr}_M(XY)$ for $X,Y\in \kk$ is non-degenerate. 
Moreover, $B_M=t_MB$, where $B$ is the Killing form. If $M$ is a
simple $\kk$-module with highest weight $\lambda$, then 
$$t_M=\frac{\mathrm{dim}M}{\mathrm{dim}\kk}(\lambda+2\rho,\lambda),$$
where $\rho$ is the half-sum of positive roots and $(\cdot,\cdot)$ is
the form on the weight lattice of $\kk$ induced by $B$. 
It is easy to check that a natural module is a simple module with minimal $t_M$.
Let $V_2$ be a natural $\gg_2$-module. Note that the restriction of $B_{V'}$ on $\gg_1$ equals $B_{V_1}$ and the restriction of
$B_{V_2}$ on $\gg_1$ equals $tB_{V_1}$ for some $t\geq 1$. On the
other hand, $t=\frac{t_{V_2}}{t_{V'}}$. Since 
$t_{V_2}$ is minimal, we have
$t=1$ and $t_{V_2}=t_{V'}$. Hence, $V'$ is a natural module, i.e. a) is proved.

To prove b), note that a classical simple Lie algebra of rank greater than $4$ admits, up to isomorphism, two (mutually dual) 
natural representations when it is of type $\sl$, and one natural representation when it is of type $\oo$ or $\sp$. Moreover, 
in the orthogonal (respectively, symplectic) case the natural module admits 
an invariant symmetric (respectively, skew-symmetric) bilinear form.

Now, assume $\gg_1\simeq \sl(i)$. We claim that $\gg_2\simeq \sl(k)$ for some $i\leq k\leq j$.
Indeed, if $\gg_2$ is not isomorphic to $\sl(k)$, then $V'$ is self-dual. Therefore its restriction 
to $\gg_1$ is self-dual, and we obtain a contradiction as $V_1$ is not a self-dual $\sl(i)$-module for $i\geq 3$.

Finally, assume $\gg_3\simeq \oo(j)$ (respectively, $\sp(2j)$). Then $V'\oplus W'$, and hence $V'$, admits an invariant symmetric 
(respectively, skew-symmetric) form. Therefore $\gg_2\simeq \oo(k)$ (respectively, $\sp(2k)$).
\end{proof}

\begin{corollary}\label{standardex}(cf. [DiP2, Corollary 2.4]) Let $\gg = \sl(V,W)$ and $\gg=\varinjlim \gg_\alpha$ for some directed system
$\{\gg_{\alpha}\}$ of simple finite-dimensional Lie subalgebras
  $\gg_\alpha \subset \gg$. Then there exists a subsystem 
$\{\gg_{\alpha'}\}$ such that
$\gg=\varinjlim \gg_{\alpha'}$ and, for every $\alpha'$,
  $\gg_{\alpha'}=\sl (V_{\alpha'},W_{\alpha'})$ 
for some finite-dimensional non-degenerate pair 
$V_{\alpha'}\subset V,W_{\alpha'}\subset W$. Similarly, if  $\gg =
\oo(V)$ 
(respectively, $\sp(V)$), then 
there exists a subsystem $\{\gg_{\alpha'}\}$ such that
$\gg=\varinjlim \gg_{\alpha'}$ and, for every $\alpha'$, $\gg_{\alpha'}=\oo (V_{\alpha'})$ 
(respectively, $\sp(V_{\alpha'})$) for some finite-dimensional non-degenerate
$V_{\alpha'}\subset V$.
\end{corollary}

\begin{proof}
 Let $\gg = \sl(V,W)$. One fixes a Lie subalgebra $\sl(V_f,W_f) \subset \gg$, where $V_f \subset V, W_f \subset W$ is a finite-dimensional non-degenerate pair, and considers the directed subsystem $\{\gg_{\alpha'}\}$ of all $\gg_{\alpha'}$ such that $\sl(V_f, W_f) \subset \gg_{\alpha'}$. There exists another finite-dimensional non-degenerate pair $V_f',W_f'$ such that $\sl(V_f,W_f)\subset \gg_{\alpha'}\subset \sl(V_f',W_f')$. Then, by Lemma \ref{lemma:1.4.10}, $\gg_{\alpha'} = \sl(V_{\alpha'},W_{\alpha'})$ for appropriate $V_{\alpha'} \subset V, W_{\alpha'} \subset W$. The cases $\gg = \oo(V), \sp(V)$ are similar.
\end{proof}

\vspace{0.33cm}

\begin{proof}\textit{ of Proposition \ref{prop:1.29.9}}
We consider the case $\gg = \sl(V,W)$ and leave the remaining cases to
the reader. Let $\gg=\sl(V,W)$ be isomorphic to 
$\sl(V',W')$. Then $\gg=\varinjlim \sl(V_f,W_f)$ over all finite-dimensional 
non-degenerate pairs $V_f\subset V, W_f\subset W$, and at the same
time $\gg=\varinjlim\sl(V'_f,W'_f)$ 
over all finite-dimensional 
non-degenerate pairs $V'_f\subset V', W'_f\subset W'$. By Corollary \ref{standardex} and Lemma \ref{lemma:1.4.10}, for each $V_f\subset V,W_f\subset W$ one can 
find $V'_f\subset V',W_f'\subset W'$ and an embedding of Lie algebras
$\sl(V_f,W_f)\subset\sl(V_f',W_f')$ 
as in Lemma \ref{lemma:1.4.10}. That implies
the existence of embeddings $V_f\hookrightarrow V_f',
W_f\hookrightarrow W_f'$ or  
$V_f\hookrightarrow W_f', W_f\hookrightarrow V_f'$ preserving the pairing. After 
a twist by transposition we may 
assume that $V_f\hookrightarrow V_f', W_f\hookrightarrow
W_f'$. Therefore we have embeddings 
$V=\varinjlim V_f\hookrightarrow V', W=\varinjlim W_f\hookrightarrow W'$ preserving the pairing. 
On the other hand, both maps are surjective since
$\sl(V',W')=\varinjlim\sl(V_f,W_f)$. 
Therefore  the linear systems $V\times W\rightarrow \CC$ and 
$V' \times W' \rightarrow \CC$ are isomorphic.
\end{proof}

\vspace{0.33cm}

Assume next that $\gg$ is an arbitrary locally finite locally simple Lie algebra. If we can choose a Cartan subalgebra $\hh_\alpha\subset\gg_\alpha$, such that $\hh_\alpha\hookrightarrow\hh_\beta$ for 
any embedding $\gg_\alpha \hookrightarrow\gg_\beta$, then $\hh:=\varinjlim\hh_\alpha$ is called a {\it local Cartan subalgebra}.

In general, a local Cartan subalgebra may not exist. For example, the following proposition
implies that the Lie algebra $\gg=\sl(V,V^*)$  does not have a local Cartan subalgebra. 

\begin{prop}\label{prop:1.4.10}
 Let $\gg = \sl(V,W)$. Then a local Cartan subalgebra of $\gg$ exists if and only if 
$V$ admits a basis $\left\lbrace v_\gamma \right\rbrace$ such that $W= \span\left\lbrace v_{\tilde\gamma}^* \right\rbrace$, 
where $v_{\tilde\gamma}^*(v_\gamma) = \delta_{{\tilde{\gamma}}\gamma}$. In this case, every local Cartan subalgebra of 
$\gg$ is of the form $\span\left\lbrace v_\gamma \otimes v_\gamma^* - v_{\tilde\gamma} \otimes v_{\tilde\gamma}^*\right\rbrace_{\gamma,\tilde\gamma}$ 
for a basis $\left\lbrace v_\gamma \right\rbrace$ as above.
\end{prop}

\begin{proof}
By Corollary \ref{standardex} we may assume $\gg=\sl(V,W)=\varinjlim
\gg_\alpha=\varinjlim\sl(V_\alpha, W_\alpha)$, 
where $V_\alpha \subset V$, $W_\alpha \subset W$ 
are certain non-degenerate finite-dimensional pairs, and that
$\hh=\varinjlim \hh_\alpha$ where $\hh_\alpha$ is a 
Cartan subalgebra of $\gg_\alpha$. Note that
for any $\alpha$ we have $\hh_\alpha\cdot V_\alpha=V_\alpha$ and  $\hh_\alpha\cdot  W_\alpha=W_\alpha$. Since $\hh$ is abelian, we have $\hh\cdot V_\alpha=V_\alpha$ and 
$\hh\cdot W_\alpha=W_\alpha$. Therefore $V$ and $W$ are semisimple $\hh-$modules.
This means that $V$ is the direct sum of non-trivial one-dimensional
$\hh-$submodules $V_\gamma$, i.e. 
$V=\bigoplus_\gamma\,V_\gamma$; similarly, 
$W = \bigoplus_{\gamma'}\,W_{\gamma'}$. 
Since however, for any $\alpha$, the spaces $V_{\alpha}$ and $W_{\alpha}$ are dual to each other, 
$\gamma'$ and $\gamma$ run over the same set of indices and
$W_\gamma(V_{\tilde\gamma}) \neq 0$ precisely for 
$\gamma = \tilde{\gamma}$. 
This yields a basis $v_\gamma$ as required: $v_\gamma$ can be chosen
as any non-zero vector in $V_\gamma$ and 
$v^*_{\gamma}$ is the unique vector in 
$W_{\gamma}$ with $v_{\gamma}^*(v_{\gamma})= 1$. 
Finally, 
$\hh = \span\left\lbrace v_\gamma \otimes v_\gamma^* - v_{\tilde\gamma} \otimes v_{\tilde\gamma}^*\right\rbrace$ as, clearly, 
$\hh\cap\gg_{\alpha}=\span\left\lbrace v_\gamma \otimes v_\gamma^* -v_{\tilde\gamma} \otimes v_{\tilde\gamma}^*\right\rbrace$ 
for $v_\gamma, v_{\tilde\gamma}\in V_\alpha$.

In the other direction, given a basis $v_\gamma$ of $V$ such that $\left\lbrace v_\gamma^* \right\rbrace$ is a basis of $W$, it is clear that 
$\gg = \varinjlim\,\sl\left(\span\left\lbrace v_\gamma\right\rbrace_{\gamma \in A}, \span\left\lbrace
v_\gamma^*\right\rbrace_{\gamma \in A}\right)$ 
for all 
finite sets of indices $A$, and that 
$\hh = \varinjlim \left(\hh \cap \span\left\lbrace v_{\gamma} \otimes v_\gamma^* - v_{\tilde\gamma} \otimes v_{\tilde\gamma}^*\right\rbrace_{\gamma,\tilde\gamma \in A}\right)$.
\end{proof}


In \cite{DPSn} (and also in earlier work, see the references in
\cite{DPSn}) \textit{Cartan subalgebras} are 
defined as maximal toral subalgebras of $\gg$ 
(i.e. as subalgebras each vector in which is
ad-semisimple). \textit{Splitting} 
Cartan subalgebras are Cartan subalgebras for which the adjoint 
representation is semisimple. It is shown in \cite{pstr} that a
countable dimensional locally finite 
locally simple Lie algebra $\gg$ admits a splitting 
Cartan subalgebra if and only if $\gg \simeq \sl(\infty), \oo(\infty),
\sp(\infty)$. Proposition \ref{prop:1.4.10} determines when Lie algebras of the form $\gg = \sl(V,W), \oo(V), \sp(V)$ admit local Cartan subalgebras and implies that the notions of local Cartan subalgebra and of splitting Cartan subalgebra coincide for these Lie algebras. 

\vspace{0.33cm}

In what follows, we denote by $V, V_*$ a pair of infinite-dimensional spaces (of not necessarily countable dimension) arising from a linear system $V \times V_* \rightarrow \CC$ for which there is a basis $\lbrace v_{\gamma}\rbrace$ of $V$ such that $V_* = \mathrm{span}(\lbrace v_{\gamma}^* \rbrace)$.

\section{The category $\mathrm{Int}_{\gg}$}\label{sect:2} 

Let $\gg$ be an arbitrary locally simple locally finite Lie algebra. An \textit{integrable} $\gg-$module is a $\gg-$module 
$M$ which is locally finite as a module over any finite-dimensional subalgebra $\gg'$ of $\gg$. In other words, $\mathrm{dim}U(\gg')\cdot m < \infty\;\;\, \forall m \in M$. We denote the category of integrable $\gg-$modules by $\mathrm{Int}_{\gg}$ : $\mathrm{Int}_{\gg}$ is a full subcategory of the category $\gg-$mod of all $\gg-$modules. It is clear that $\mathrm{Int}_{\gg}$ is an abelian category and a monoidal category with respect to usual tensor product. Note that the adjoint representation of $\gg$ is an object of $\mathrm{Int}_{\gg}$.

The \textit{functor of $\gg-$integrable vectors}
\begin{eqnarray}\label{func_int_vec}
 \Gamma_{\gg} &:& \gg-\mathrm{mod} \rightsquigarrow \mathrm{Int}_{\gg},\\
\Gamma_{\gg}(M) &:=& \left\lbrace m \in M\; | \;\mathrm{dim}U(\gg')\cdot m < \infty\; \text{for all finite-dimensional subalgebras}\; \gg'\subset \gg\right\rbrace
\end{eqnarray}
is a well defined left-exact functor. This follows from the fact that the functor of $\gg'-$finite vectors $\Gamma_{\gg'}$ is well defined for any finite-dimensional subalgebra $\gg' \subset \gg$, see for instance \cite{Z}, and that $\gg$ equals the direct limit of its finite-dimensional subalgebras. 

\begin{theo}\label{int1}
a) Let $M$ be an object of $\mathrm{Int}_{\gg}$. Then $\Gamma_{\gg}(M^*)$ is an injective object of $\mathrm{Int}_{\gg}$.

b) $\mathrm{Int}_{\gg}$ has enough injectives. More precisely, for any object $M$ of $\mathrm{Int}_{\gg}$ there is a canonical injective homomorphism of $\gg-$modules
\begin{equation}\label{can_inj_hom}
 M \rightarrow \Gamma_{\gg} (\Gamma_{\gg}(M^*)^*).
\end{equation}
\end{theo}
\begin{proof} In \cite{ps}, see Proposition $3.2$ and Corollary $3.3$, the proof is given under the assumption that $\gg$ is countable dimensional. The reader can check that this assumption is inessential.
\end{proof}

\vspace*{0.5cm}
\section{Five subcategories of $\mathrm{Int}_{\gg}$}\label{sect:3}

\subsection{The category $\mathrm{Int}_{\gg}^\alg$}\label{sect:3.1}
We start by defining the full subcategory $\mathrm{Int}_{\gg}^\alg \subset \mathrm{Int}_{\gg}$. Its objects are integrable $\gg-$modules $M$ such that, for any simple finite-dimensional subalgebra $\gg'\subset\gg$, the restriction of $M$ to $\gg'$ is a direct sum of finitely many $\gg'-$isotypic components. Clearly, if $\mathrm{dim}\;M=\infty$, at least one of these isotypic components must be infinite dimensional. If $\gg$ is diagonal, the adjoint representation of $\gg$ is easily seen to be an object of $\mathrm{Int}_{\gg}^\alg$.  

The following proposition provides equivalent definitions of $\mathrm{Int}_{\gg}^\alg$. 

\begin{prop}\label{alg1} a) $M\in\mathrm{Int}_{\gg}^\alg$ iff $M$ and $M^*$ are integrable.

b) An integrable $\gg-$module $M$ is an object of $\mathrm{Int}_{\gg}^\alg$ iff for any $X\in\gg$ there exists a non-zero polynomial $p(t)\in\CC[t]$ such that $p(X)\cdot M=0$.
\end{prop}
\begin{proof} a) In the countable-dimensional case the statement is proven in [PS, Lemma 4.1].
In general, let $\gg'\subset\gg$ be a finite-dimensional simple subalgebra and $M=\oplus_\alpha M_\alpha$ be the decomposition of $M$ into $\gg'$-isotypic components. Then it is straightforward to check that $M^*=\prod_\alpha M_\alpha^*$ is an integrable $\gg'$-module iff the direct product is finite. This proves a), since a $\gg-$module is integrable iff it is $\gg'-$integrable for all finite-dimensional Lie subalgebras $\gg' \subset \gg$.

b) Let $M\in \mathrm{Int}_{\gg}^\alg$. Any $X\in\gg$ lies in some finite-dimensional Lie subalgebra $\gg' \subset \gg$. For each $\gg'$-isotypic component $M_i$ of $M$ there exists $p_i(t)$ such that $p_i(X) \cdot M_i = 0$. Since there are finitely many $\gg'-$isotypic components, we can set $p(t)=\prod_i p_i(t)$. Then $p(X) \cdot M = 0$. 

On the other hand, if $M\notin  \mathrm{Int}_{\gg}^\alg$ then there are infinitely many isotypic components for some finite-dimensional simple $\gg'\subset\gg$. That implies the existence of a semisimple $X\in \gg'$ which has infinitely many eigenvalues in $M$. Therefore $p(X) \cdot M \neq 0$ for any $0 \neq p(t) \in \CC[t]$. 
\end{proof}

It is obvious that  $\mathrm{Int}_{\gg}^\alg$ is an abelian monoidal subcategory of $\gg-$mod. It is also closed under dualization. 

\begin{prop}\label{alg2}  $\mathrm{Int}_{\gg}^\alg$ contains a non-trivial module iff $\gg$ is diagonal.
\end{prop}
\begin{proof} Again, for a countable dimensional $\gg$ the statement is proven in \cite{ps} (see Proposition 4.3). In fact, we prove in \cite{ps} that if $\gg=\varinjlim \gg_i$ has a non-trivial integrable module such that $M^*$ is also integrable, then the embedding  $\gg_i\hookrightarrow\gg_{i+1}$ is diagonal for all sufficiently large $i$.

To give a general proof, it remains to show that if $\gg$ is not diagonal, then $\Int_{\gg}^\alg$ contains no non-trivial modules. Assume that $\gg=\varinjlim \gg_{\alpha}$ is not diagonal. Fix a simple finite-dimensional Lie algebra $\gg_{\alpha_1}$ and a simple $\gg-$module $M \in \mathrm{Int}_{\gg}^\alg$ such that $M_{\downarrow \gg_{\alpha_1}}$ is non-trivial. We claim that one can find a chain of proper embeddings of simple finite-dimensional Lie algebras $$\gg_{\alpha_1} \hookrightarrow \gg_{\alpha_2} \hookrightarrow \cdots \hookrightarrow \gg_{\alpha_i} \hookrightarrow \gg_{\alpha_{i+1}} \hookrightarrow \cdots$$ such that the embeddings $\gg_{\alpha_i} \hookrightarrow \gg_{\alpha_{i+1}}$ are not diagonal. Indeed, otherwise there will exist $\beta_0$ so that the embedding $\gg_{\beta_0} \hookrightarrow \gg_{\alpha}$ is diagonal for all $\alpha > \beta_0$. Then, since $\gg = \varinjlim_{\alpha > \beta_0}\,\gg_{\alpha}$, $\gg$ is diagonal. This shows that the existence of $\beta_0$ is contradictory. 
Now Proposition 4.3 in \cite{ps} implies that $M_{\downarrow\varinjlim\,\gg_{\alpha_i}}$ is a trivial 
module, which shows that the assumption that $M_{\downarrow\gg_{\alpha_1}}$ is non-trivial is false.
%
\end{proof}

\vspace{0.5cm}

Let $\gg = \sl(V,W)$ (respectively, $\gg = \oo(V), \sp(V)$). Then the tensor products $T^{m,n}:=V^{\otimes m}\otimes W^{\otimes n}$ (respectively, $T^m := V^{\otimes m}$) and their simple subquotients are objects of $\mathrm{Int}_{\gg}^\alg$. 

Here is a less trivial example of a simple object of $\mathrm{Int}_{\gg}^\alg$ for $\gl = \sl(V,V_*)$ where $V$ is a countable-dimensional vector space. Let $\gg =\varinjlim \gg_i$ where $\gg_i = \sl(V_{i}), \mathrm{dim}\,V_i = i+1$, and $\varinjlim V_{i} = V$. Fix $s\in\mathbb Z_{\geq 1}$. Define $\Lambda^{\infty-s}\,V$ as the direct limit $\varinjlim\Lambda^{i+1-s}(V_i)$ (where $\Lambda^{i+1-s}(V_{i})=0$ for $i+1\leq s$). Then $\Lambda^{\infty-s}\,V$ is a simple object of $\mathrm{Int}_{\gg}^\alg$. 

Given a $\gg-$module $M \in \Int_{\gg}^\alg$, where $\gg = \varinjlim\,\gg_{\alpha}$, for each $\alpha$ we can assign to $\gg_{\alpha}$ the finite set of isomorphism classes of simple finite-dimensional $\gg_{\alpha}-$modules which occur in the restriction $M_{\downarrow\gg_{\alpha}}$. A. Zhilinskii has defined a \textit{coherent local system of finite-dimensional representations} of $\gg = \varinjlim\,\gg_{\alpha}$ as a function of $\alpha$ with values in the set of isomorphism classes of finite-dimensional $\gg_{\alpha}-$modules, with the following compatibility condition: if $\beta < \alpha$ then the representations assigned to $\beta$ are obtained by restriction from the representations assigned to $\alpha$. Thus, every $M \in \Int_{\gg}^\alg$ determines a coherent local system of \textit{finite type}, i.e. a local system containing finitely many isomorphism classes for any $\alpha$.

Zhilinskii has classified all coherent local systems under the condition that $\gg$ is countable dimensional \cite{Zh1}, \cite{Zh2} (see also  \cite{PP} for an application of Zhilinskii's result). In particular, he has proved that proper coherent local systems, i.e. coherent local systems different from the ones assigning the trivial $1-$dimensional module to all $\alpha$, or all finite-dimensional $\gg_{\alpha}-$modules to $\alpha$, exist only if $\gg$ is diagonal. This leads to another proof of Proposition \ref{alg2}.

The category  $\mathrm{Int}_{\gg}^\alg$ has enough injectives as follows immediately from Proposition \ref{alg1} a) and Theorem \ref{int1}.
We know of no classification of simple modules in $\Int_{\gg}^\alg$.

\subsection {The category  $\mathrm{Int}^{\wt}_{\gg,\hh}$}\label{sect:3.2}
Given a local Cartan subalgebra $\hh \subset \gg$, we define 
$\mathrm{Int}^{\wt}_{\gg,\hh}$ as the full subcategory of $\mathrm{Int}_{\gg}$ 
consisting of \textit{$\hh-$semisimple} integrable $\gg-$modules, i.e. integrable $\gg-$modules $M$ admitting an $\hh-$weight decomposition
\begin{equation}\label{M_weight_dec}
 M = \oplus_{\lambda \in \hh^*} M^{\lambda}
\end{equation}
where 
\begin{equation}\label{eq.6}
M^{\lambda} := \left\lbrace m \in M\; |\; h \cdot m = \lambda(h)m\;\; \forall h \in \hh\right\rbrace. 
\end{equation}

If $\gg = \sl(V,W), \oo(V), \sp(V)$ for countable-dimensional $V,W$, then $V$ (and $W$ in case $\gg = \sl(V,W)$) is a simple object of $\Int^{\wt}_{\gg,\hh}$ for any $\hh$. Moreover, if $\gg$ is a countable-dimensional locally simple Lie algebra, it is proved in \cite{pstr} that the adjoint representation of $\gg$ is an object of $\Int^{\wt}_{\gg,\hh}$ iff $\gg \simeq \sl(\infty), \oo(\infty), \sp(\infty)$. The simple modules of $\Int^{\wt}_{\gg,\hh}$ for $\gg = \sl(\infty),\oo(\infty),\sp(\infty)$ have been studied in \cite{dip}, however there is no classification of such modules.

Assume that $\gg$ is a locally simple diagonal countable-dimensional Lie algebra. Without loss of generality, assume that $\gg = \varinjlim\,\gg_i$, where all $\gg_i$ are of the same type $A,B,C,$ or $D$. The very definition of $\gg$ implies that there is a well-defined chain

\begin{equation}\label{inf_ch_nat}
 V_{\oj_{1}} \stackrel{\ae_1}{\hookrightarrow} V_{\oj_{2}} \stackrel{\ae_2}{\hookrightarrow} \dots \hookrightarrow V_{\oj_{i}} \stackrel{\ae_i}{\hookrightarrow} V_{\oj_{i+1}} \hookrightarrow \dots
\end{equation}
of embeddings of natural $\oj_i$-modules, and we call its direct limit $V$ a \textit{natural representation} of $\oj$. Moreover, a fixed natural representation $V$ is a simple object of $\Int^{\wt}_{\gg,\hh}$ for some local Cartan subalgebra $\hh$. To see this, we use induction to define a local Cartan subalgebra $\hh \subset \gg$ so that $V \in \Int^{\wt}_{\gg,\hh}$. Given $\hh_i \subset \gg_i$ and an $\hh_i-$eigenbasis $\textbf{b}_{i}$ of $V_{i}$, let $\hh_{i+1}$ be a Cartan subalgebra of $\gg_{i+1}$ whose eigenbasis $\textbf{b}_{i+1}$ of $V_{i+1}$ contains $\textbf{b}_i$. The assumption that $\gg_i$ and $\gg_{i+1}$ are of the same type $A,B,C$ or $D$ implies that $\hh_{i+1}$ exists as required. Moreover, $\hh:=\varinjlim\,\hh_i$ is a well-defined local Cartan subalgebra of $\gg$ and  $V \in \Int^{\wt}_{\gg,\hh}$.

Assume next that $\gg$ is a locally simple Lie algebra which admits a
local Cartan subalgebra $\hh$ such that the adjoint representation
belongs to  $\Int^{\wt}_{\gg,\hh}$. This certainly holds for 
$\gg = \sl(\infty),\oo(\infty),\sp(\infty)$, but also for instance for 
$\gg = \sl(V,V_*)$ where $V$ is an arbitrary vector space. 
In this case we can define a left exact functor 
$\Gamma_\hh^\wt:\mathrm{Int}_{\gg}\rightsquigarrow\mathrm{Int}^{\wt}_{\gg,\hh}$ by setting
\begin{equation}\label{eq.7}
\Gamma_\hh^\wt(M):=\oplus_{\lambda\in\hh^*} M^{\lambda},
\end{equation}
where $ M^{\lambda}$ is given by (\ref{eq.6}). It is easy to see that $\Gamma_\hh^\wt$ is right adjoint to the inclusion functor $\mathrm{Int}^{\wt}_{\gg,\hh}\rightsquigarrow\mathrm{Int}_{\gg}$. Hence $\Gamma_\hh^\wt$ maps injectives to injectives, and therefore $\mathrm{Int}^{\wt}_{\gg,\hh}$ has enough injectives. We do not know whether $\Int^{\wt}_{\gg,\hh}$ has enough injectives when the adjoint representation is not an object of $\mathrm{Int}^{\wt}_{\gg,\hh}$.

We conjecture that for non-diagonal Lie algebras $\gg$ the category $\mathrm{Int}^{\wt}_{\gg,\hh}$ consists of trivial modules only. 

\subsection {The category  $\mathrm{Int}_{\gg,\hh}^\fin$}\label{sect:3.3}

By $\mathrm{Int}_{\gg,\hh}^\fin$ we denote the full subcategory of $\mathrm{Int}^{\wt}_{\gg,\hh}$ consisting of integrable $\gg-$modules satisfying 
$\mathrm{dim}M^{\lambda} < \infty \;\;\forall \lambda \in \hh^*$. 


Note that for $\gg = \sl(V,V_*)$ (respectively, for $\gg = \oo(\infty), \sp(\infty)$) the tensor products $T^{m,0} = V^{\otimes m}$ and $T^{0,n} = W^{\otimes n}$ (respectively, $T^m = V^{\otimes m}$) are objects of $\Int_{\gg,\hh}^\fin$ for every local Cartan subalgebra $\gg$. However, the adjoint representation is not in $\Int_{\gg,\hh}^\fin$ for any $\hh$.

If $\gg$ is countable dimensional diagonal then, as shown above, for each natural representation $V$ there is a local Cartan subalgebra $\hh$ so that $V$ (and more generally $V^{\otimes m}$) is an object of $\Int^{\wt}_{\gg,\hh}$. In fact, $V^{\otimes m} \in \Int_{\gg,\hh}^\fin$ for any $m \geq 0$. 

Here is a more interesting example of a simple module in $\Int_{\gg,\hh}^\fin$ for $\gg = \sl(V,V_*)$, where $V$ is a countable-dimensional vector space. Fix a chain of embeddings $$
\gg_1 \hookrightarrow \gg_2 \hookrightarrow \cdots \hookrightarrow \gg_i \hookrightarrow \gg_{i+1} \hookrightarrow \cdots$$ so that $\gg = \sl(V_{i})$ for $\mathrm{dim}\,V_i = i+1, V = \varinjlim V_i, \gg = \varinjlim \gg_i$. Note that there is a canonical injection of $\gg_i-$modules $S^{i+1}(V_{i}) \hookrightarrow S^{i+2}(V_{i+1})$, and set $\Delta:=\varinjlim\,S^{i+1}(V_i)$. Then one can check that $\Delta$ is a multiplicity free $\hh-$module, where $\hh$ is such that $\hh_i := \hh \cap \gg_i$ is a Cartan subalgebra of $\gg_i$.

The following result is proved in \cite{ps}.

\begin{prop}\label{prop_ps} Let  $\gg=\sl(\infty),\oo(\infty),\sp(\infty)$.
Then the category $\mathrm{Int}_{\gg,\hh}^\fin$ is semisimple.  
\end{prop}
This result should be considered an extension of Weyl's semisimplicity theorem to the case of direct limit Lie algebras. It is an interesting question whether the category $\mathrm{Int}_{\gg,\hh}^\fin$ is semisimple whenever it is well defined. 

\subsection{The category $\widetilde{\mathrm{Tens}}_{\gg}$}\label{sect:3.4}
Let $M$ be a $\gg-$module. 
Recall that the \textit{socle} $\Soc\,M = \Soc^1\,M$ of $M$ is the unique maximal semisimple submodule of 
$M$, and $\mathrm{soc}^kM:=\pi^{-1}(\mathrm{soc}(M/\mathrm{soc}^{k-1}M))$ for $k \geq 2$, where $\pi : M \rightarrow M/\mathrm{soc}^k M$ is the natural projection. The filtration 
\begin{equation}\label{eq.filt}
 0 \subset \Soc\,M = \Soc^1\,M \subset \Soc^2\,M\subset \dots\subset\Soc^k\,M\subset\dots
\end{equation}
is by definition the \textit{socle filtration} of $M$.
The $\gg-$module $M$ has \textit{finite Loewy length} if it has a finite and exhaustive socle filtration, i.e.
\begin{equation}\label{soc_filt}
 M = \mathrm{soc}^l M
\end{equation}
for some $l$. 

By definition, $\widetilde{\mathrm{Tens}}_{\gg}$ is the full subcategory of $\mathrm{Int}_{\gg}$ whose objects are integrable $\gg-$modules with the property that both $M$ and $\Gamma_{\gg}(M^*)$ have finite Loewy length.

The category $\widetilde{\mathrm{Tens}}_{\gg}$ is studied in detail in
\cite{ps} for 
$\gg = \sl(\infty), \oo(\infty), \sp(\infty)$, where it is shown in
particular that 
$\Gamma_{\gg}(M^*) = M^*$ for any object $M$ of
$\widetilde{\mathrm{Tens}}_{\gg}$. A major result of \cite{ps} is
that, up to isomorphism, the simple objects of
$\widetilde{\mathrm{Tens}}_{\gg}$ are precisely the simple
subquotients of the tensor algebra 
$T(V\oplus V_*)$ for $\gg = \sl(V,V_*) \simeq \sl(\infty)$, and of the
tensor algebra 
$T(V)$ for $\gg = \oo(V) \simeq \oo(\infty)$ or $\gg = \sp(V) \simeq\sp(\infty)$. 
These simple modules are discussed in more detail in section
\ref{sect:4} below. Note that the objects of 
$\widetilde{\mathrm{Tens}}_{\gg}$ have in general infinite length and are not objects of $\mathrm{Int}^{\wt}_{\gg,\hh}$ for any $\hh$. 
An example of infinite length module in
$\widetilde{\mathrm{Tens}}_{\gg}$ for $\gg = \sl(V,V_*) \simeq\sl(\infty)$ is 
$V^*$: there is a non-splitting exact sequence of $\gg-$modules 
$$ 0 \to V_* = \Soc\,V^* \rightarrow V^* \rightarrow V^*/V_*\rightarrow 0 $$ and $V^*/V_*$ is a 
trivial module of uncountable dimension.

For $\gg = \sl(\infty), \oo(\infty), \sp(\infty)$, the category
$\widetilde{\mathrm{Tens}}_\gg$ 
has enough injectives [PS, Corollary 6.7a)].

\subsection{The category $\TT_{\gg}$}\label{sect:3.5}
The fifth subcategory we would like to introduce in this section is
the category of tensor modules 
$\TT_{\gg}$. We define this category only for $\gg=\sl(V,W),\oo(V),\sp(V)$, and discuss it in detail in section \ref{sect:5}. 

We call a subalgebra $\kk\subset \sl(W,V)$  a \textit{finite-corank
 subalgebra} if it contains the subalgebra 
$\sl(W_0^\perp,V_0^\perp)$ for some finite-dimensional non-degenerate pair
$V_0\subset V,W_0\subset W$. Similarly, we call $\kk\subset \oo(V)$
(respectively, 
$\sp(V)$) a \textit{finite corank subalgebra}  if it contains $ \oo(V^\perp)$ (respectively, $\sp(V_0^\perp)$) for some finite-dimensional $V_0\subset V$ such that the restriction of the form on $V_0$ is non-degenerate.

We say that a $\gg$-module $L$ {\it satisfies the large annihilator
  condition} if the annihilator in $\gg$ of any 
$l\in L$ contains a finite-corank subalgera. It follows immediately
from definition that if $L_1$ and $L_2$ 
satisfy the large annihilator condition, then the same holds also for $L_1\oplus L_2$ and $L_1\otimes L_2$.

By $\TT_{\gg}$ we denote the category of finite length integrable $\gg-$modules which satisfy the large annihilator condition. By definition, $\TT_{\gg}$ is a full subcategory of $\mathrm{Int}_{\gg}$. It is clear that $\TT_\gg$ is a monoidal category with respect to usual tensor product $\otimes$.

\subsection{Inclusion pattern}\label{sect:3.6} 
The following diagram summarizes the inclusion pattern for the five subcategories of $\mathrm{Int}_{\gg}$ introduced above:


\begin{figure}[H]
 \centering
\includegraphics[scale=.525]{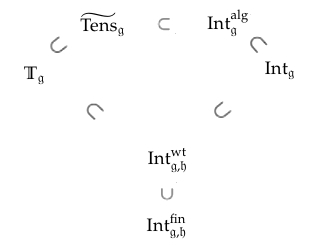}.
\begin{equation}\label{diagram}
\end{equation}
\end{figure} Note that all categories except $\TT_\gg$ are defined for any locally simple Lie algebra $\gg$, while $\TT_\gg$ is defined only for $\gg = \sl(V,W), \oo(V), \sp(V)$. Moreover, under the latter assumption all inclusions are strict. We support this claim by a list of examples and leave it to the reader to complete the proof. 


\paragraph{\textbf{Examples}}

  Let $\gg = \sl(V,V_*), \oo(V), \sp(V)$, where $V$ is countable dimensional. The simple objects of $\TT_\gg$ and $\widetilde{\mathrm{Tens}}_\gg$ are the same, however $V^* \in \widetilde{\mathrm{Tens}}_\gg$ while $V^* \notin \TT_\gg$. Moreover, $V^* \notin \mathrm{Int}^{\wt}_{\gg,\hh}$ for any local Cartan subalgebra $\hh$. The module $\Delta$ from subsection \ref{sect:3.3} is an object of $\mathrm{Int}_{\gg,\hh}^\fin$ but not an object of $\mathrm{Int}_\gg^\alg$.  The adjoint representation is an object of $\mathrm{Int}^{\wt}_{\gg,\hh}$ but not of $\Int_{\gg,\hh}^\fin$.

\section{Mixed tensors}\label{sect:4}
In this section $\gg = \sl(V,W), \oo(V), \sp(V)$. By definition, $V$ is a $\gg-$module. For $\gg = \sl(V,W), W$ is also a $\gg-$module.

Consider the tensor algebra $T(V)$ of $V$. Then usual finite-dimensional Schur duality implies that
\begin{equation}\label{schur_imp}
 T(V)=\bigoplus_{\lambda}\CC_{\lambda} \otimes V_{\lambda}, 
\end{equation}
where $\lambda$ runs over all Young diagrams (i.e. over all partitions of all integers $m \in \ZZ_{\geq 0}$), $\CC_{\lambda}$ denotes the irreducible $S_{|\lambda|}-$module (where $|\lambda|$ is the degree of $\lambda$) corresponding to $\lambda$, and $V_{\lambda}$ is the image of the Schur projector corresponding to $\lambda$. For $\gg = \sl(V,W), V_{\lambda}$ is an irreducible $\gg-$module as it is isomorphic to the direct limit $\varinjlim(V_f)_{\lambda}$ of the directed system $\{(V_f)_\lambda\}$ of irreducible $\sl(V_f,W_f)-$modules for sufficiently large non-degenerate finite-dimensional pairs $V_f \subset V, W_f \subset W$. For $\gg=\oo(V), \sp(V), V_{\lambda}$ is in general a reducible $\gg-$module.

Similarly, for $\gg = \sl(V,W),$
\begin{equation}\label{schur_imp1}
 T(W)=\bigoplus_{\lambda}\CC_{\lambda} \otimes W_{\lambda}.
\end{equation}

Let $\gg = \sl(V,W)$. Recall that $T^{m,n}=V^{\otimes m}\otimes W^{\otimes n}$.
Then
\begin{equation}
 T^{m,n} = \bigoplus_{|\lambda|=n, \; |\mu| = m} \CC_{\lambda} \otimes \CC_{\mu} \otimes V_{\lambda} \otimes W_{\mu}.
\end{equation}

Note that, as a $\gg-$module $T(V,W):=\bigoplus_{m,n\geq 0} T^{m,n}$ is not completely reducible. This follows simply from the observation that the exact sequence
\begin{equation}\label{seqv}
 0 \rightarrow \gg \rightarrow V \otimes W \rightarrow \CC \rightarrow 0
\end{equation}
does not split as $V \otimes W$ has no trivial submodule. In \cite{pstyr} the structure of $T(V,W)$ has been studied in detail for 
countable-dimensional $V$ and $W$.

For each ordered set 
$I=\left\lbrace i_1, \dots, i_k, j_1, \dots, j_k\right\rbrace$, where $i_1, \dots, i_k \in \left\lbrace1,\dots,m\right\rbrace, j_1,\dots,j_k \in \left\lbrace1,\dots,n\right\rbrace, k \leq \mathrm{min}\left\lbrace m,n \right\rbrace,$ 
there is a well-defined morphism of $\gg-$modules
\begin{equation}\label{comp_contr}
 \varphi_I\; : \; T^{m,n} \longrightarrow T^{m-k,n-k}
\end{equation}
which is simply the composition of contractions
\begin{eqnarray}\label{contr}
 T^{m,n} \xrightarrow{\varphi_{i_1,j_1}} T^{m-1,n-1} \xrightarrow{\varphi_{i_2,j_2}} T^{m-2,n-2} \xrightarrow{\varphi_{i_3,j_3}} &\dots& \xrightarrow{\varphi_{i_k,j_k}} T^{m-k,n-k}.
\end{eqnarray}
 Here
\begin{eqnarray}
\varphi_{i,j} \; : \; T^{m,n} \longrightarrow T^{m-1,n-1}.
\end{eqnarray}
denotes the contraction $V \otimes W \rightarrow \CC$ where $V$ is identified with the $i-$th factor in $V^{\otimes m}$ and $W$ is identified with the 
$j-$th factor in $W^{\otimes n}$.

We now define a filtration of $T^{m,n}$ by setting
\begin{equation}\label{gen_soc}
F^{m,n}_0:=0,\,\,F^{m,n}_k:= \cap_I\, \mathrm{ker}\, \varphi_I\,\text{for}\,k=1,\dots,\min\lbrace m,n\rbrace,\,F^{m,n}_{\min \lbrace m,n\rbrace + 1}:=T^{m,n}, 
\end{equation}
where $I$ runs over all ordered sets 
$\left\lbrace i_1,\dots,i_{k},j_1,\dots,j_{k}\right\rbrace$ as above.

Let $|\lambda|=m,|\mu|=n$. We set
$$V_{\lambda,\mu}:=F^{m,n}_1\cap  (V_{\lambda} \otimes W_{\mu}).$$ 
Note that, for sufficiently large finite-dimensional non-degenerate pairs $V_f\subset V, W_f\subset W$, the $\sl(V_f,W_f)$-module $T(V_f,W_f)\cap V_{\lambda,\mu}$ is simple.
Therefore $V_{\lambda,\mu}$ is a simple $\sl(V,W)$-module.

\begin{theo}\label{soc-sl} $\lbrace F^{m,n}_k\rbrace_{0\leq k\leq \min\lbrace m,n\rbrace+1}$ is the socle filtration of $T^{m,n}$ as a $\sl(V,W)-$module.
\end{theo}

\begin{proof}
In \cite{pstyr} this theorem is proven in the countable-dimensional case. Here we give a proof for arbitrary $V$ and $W$.

Recall that if $M$ is a $\gg-$module, $M^\gg$ stands for the space of $\gg-$invariants in $M$.

\begin{lemma}\label{inv} Let $\gg=\sl(V,W)$ (respectively, $\oo(V)$ or $\sp(V)$). Then $(T^{m,n})^\gg=0$ for $m+n>0$ 
(respectively, $(T^m)^\gg=0$ for $m>0$).
\end{lemma}
\begin{proof} We prove the statement for $\gg=\sl(V,W)$ and $m>0$. The other cases are similar. Let $u\in T^{m,n} = V^{\otimes m}\otimes W^{\otimes n}$, $u\neq 0$. Then  $u\in V_f^{\otimes m}\otimes W_f^{\otimes n}$ for some finite-dimensional non-degenerate pair $V_f\subset V, W_f\subset W$. Choose bases in $V_f$ and $W_f$ and write $$u=\sum_{i=1}^tc_i v^i_1\otimes\dots\otimes v^i_m\otimes w^i_1\otimes\dots\otimes w^i_n$$ where all $v^i_j$ and $w^i_j$ are basis vectors respectively of $V_f$ and $W_f$. Pick $w\in W$ such that $\tr(v^1_1 \otimes w)=1$ and $\tr(v^i_j \otimes w)=0$ for all $v^i_j\neq v^1_1$. Let $v\in V \backslash V_f$ and $w\in W_f^\perp$. Then $$(v \otimes w)\cdot u =\sum_{i=1}^t\sum_{j=1}^mc_i \tr(v^i_j\otimes w) v^i_1\otimes\dots\otimes v^i_{j-1}\otimes v\otimes v^i_{j+1}\otimes\dots\otimes v^i_m\otimes w^i_1\otimes\dots\otimes w^i_n.$$ Our choice of $v$ and $w$ ensures that at least one term in the right-hand side is not zero and there is no repetition in the tensor monomials appearing 
with non-zero 
coefficient. That implies $(v \otimes w)\cdot u \neq 0$. Hence $u\notin  (V^{\otimes m}\otimes W^{\otimes n})^\gg$.
\end{proof}

\begin{lemma}\label{soclemma1} Let $\gg = \sl(V,W)$. If $\Hom_\gg(V_{\lambda,\mu}, T^{m,n})\neq 0$ then $|\lambda|=m,|\mu|=n$.
\end{lemma}

\begin{proof} Choose a finite-dimensional non-degenerate pair $V_f\subset V, W_f\subset W$ such that $\dim V_f\geq \operatorname{max}\lbrace m,n,|\lambda|,|\mu|\rbrace$. Then $(V_f)_{\lambda,\mu}:=T(V_f,W_f)\cap V_{\lambda,\mu}$ is annihilated by the finite corank subalgebra $\kk=\sl(W_f^\perp,V_f^\perp)$ of $\gg$. Let $\mathfrak l=\sl(V_f,W_f)\oplus \kk$. Then $$\Hom_{\mathfrak l}((V_f)_{\lambda,\mu},T^{m,n})=\Hom_{\sl(V_f,W_f)}((V_f)_{\lambda,\mu},(T^{m,n})^\kk)=\Hom_{\sl(V_f,W_f)}((V_f)_{\lambda,\mu},V_f^{\otimes m}\otimes W_f^{\otimes n}).$$ Therefore a homomorphism $\varphi\in\Hom_\gg(V_{\lambda,\mu}, T^{m,n})$ has a well-defined restriction $\varphi_f\in \Hom_{\sl(V_f,W_f)}((V_f)_{\lambda,\mu},V_f^{\otimes m}\otimes W_f^{\otimes n})$. According to finite-dimensional representation theory, $\phi_f\neq 0$ implies that $\varphi_f$ is a composition 
$$(V_f)_{\lambda,\mu}\to (V_f^{\otimes|\lambda|} \otimes W_f^{\otimes|\mu|})\otimes(V_f^{\otimes(m-|\lambda|)}\otimes W_f^{\otimes(n-|\mu|)})^{\sl(V_f,W_f)}\xrightarrow{} V_f^{\otimes m}\otimes W_f^{\otimes n}.$$ 
Since $\varphi$ is the inverse limit of $\varphi_f$, $\varphi$ is a composition 
$$V_{\lambda,\mu} \to T^{|\lambda|,|\mu|}\otimes  (T^{m-|\lambda|,n-|\mu|})^{\sl(V,W)}\xrightarrow{} T^{m,n}.$$ 
However, by Lemma \ref{inv}, $(T^{m-|\lambda|,n-|\mu|})^{\sl(V,W)}\neq 0$ only if  $|\lambda|=m,|\mu|=n$.
\end{proof}

Note that Lemma \ref{soclemma1} implies 
\begin{equation}\label{socle:filt1}
 \mathrm{soc}T^{m,n}=\mathrm{soc}^1T^{m,n}=F_1^{m,n}.
\end{equation}
Consider now the exact sequence
\begin{equation}\label{socle:filt2}
 0\to F^{m,n}_{k-1}\to T^{m,n}\to \bigoplus_I T^{m-k+1,n-k+1},
\end{equation}
where $I$ runs over the same set as in (\ref{gen_soc}).
It follows from (\ref{socle:filt1}) that (\ref{socle:filt2}) induces an exact sequence
$$0\to F^{m,n}_{k-1}\to F^{m,n}_{k}\to \bigoplus_I F^{m-k+1,n-k+1}_1.$$
Therefore induction on $k$ yields $\operatorname{soc}^{k}T^{m,n}=F^{m,n}_{k}$.
Theorem \ref{soc-sl} is proved.
\end{proof}

As a corollary we obtain that the $\sl(V,W)-$module $V_{\lambda}
\otimes W_{\mu}$ is indecomposable since its socle 
$V_{\lambda,\mu}$ is simple. Further one shows that any simple
subquotient of $T(V,W)$ is isomorphic to 
$V_{\lambda,\mu}$ for an appropriate pair of partitions $\lambda,
\mu.$ In the socle filtration of $V_{\lambda} \otimes W_\mu$, the
$k-$th level, i.e. the quotient $\mathrm{soc}^k(V_{\lambda} \otimes W_\mu)/\mathrm{soc}^{k-1}(V_{\lambda}\otimes W_\mu)$, can 
have only simple constituents isomorphic to $V_{\lambda',\mu'}$ where $\lambda'$ is obtained from $\lambda$ by removing $k-1$ boxes and $\mu'$ is obtained from $\mu$ by removing $k-1$ boxes. An explicit formula for the multiplicity of $V_{\lambda' \mu'}$ in  $\mathrm{soc}^k(V_{\lambda} \otimes W_{\mu})/\mathrm{soc}^{k-1}(V_{\lambda}\otimes W_\mu)$ is given in \cite{pstyr}.

Next, consider the algebra $\mathcal
A_{\sl(V,W)}\subset \End_{\sl(V,W)}(T(V,W))$ generated by all
contractions $\phi_{i,j}$ and 
by the direct sum of group algebras $\bigoplus_{m,n\geq  0}\CC[S_m\times S_n]$. It is clear that $\mathcal A_{\sl(V,W)}$ does
not depend on the choice of the linear system $V \times W \rightarrow\CC$. 
In what follows we use the notation $\mathcal A_{\sl}$. One can equip
$\mathcal A_{\sl}$ with a 
$\ZZ_{\geq 0}$-grading  $\mathcal A_{\sl}=\bigoplus_{q\geq 0}
(\mathcal A_{\sl})_q$ by setting 
$(\mathcal A_{\sl})_q:=\bigoplus_{m,n\geq
  0}{\mathrm{Hom}_{\sl(V,W)}}(T^{m,n},T^{m-q,n-q})\cap A_{\sl}$. If we
set 
$T^{\leq r}(V,W):=\bigoplus_{m+n\leq r} T^{m,n}$ and denote by
$\mathcal A_{\sl}^{(r)}$ the intersection of  
$\mathcal A_{\sl}$ with $\End_{\sl(V,W)}(T^{\leq r}(V,W))$, then,
obviously,   
$\mathcal A_{\sl}=\varinjlim \mathcal A_{\sl}^{(r)}$.

The following statement is a central result in \cite{DPS}.

\begin{prop}\label{Koszul1} 
\begin{enumerate}[a)]
 \item If $V$ is countable dimensional, then 
$(\mathcal A_{\sl})_q = \bigoplus_{m,n \geq 0}\,\Hom_{\sl(V,V_*)}\,(T^{m,n},T^{m-q,n-q})$.
 \item $\mathcal A_{\sl}^{(r)}$ is a Koszul self-dual ring for any $r \geq 0$.
\end{enumerate}
\end{prop}

Now let $\gg=\oo(V)$ (respectively, $\sp(V)$). Recall that $T^m =
V^{\otimes m}$. Assume $m\geq 2$. For a pair of indices 
$1\leq i<j\leq m$ we have a contraction map
$\varphi_{i,j}\in\Hom_\gg(V^{\otimes m},V^{\otimes m-2})$. 
If $V$ is countable dimensional, the socle filtration of $T(V)$
considered as a $\gg-$module is described in
\cite{pstyr}. Recall the decomposition (\ref{schur_imp}). Each
$V_\lambda$ is an indecomposable $\gg$-module with simple socle which 
we denote by $V_{\lambda,\gg}$. Moreover, 
$$\Soc^k\,V_{\lambda} = \Soc^k\,(V_{\lambda} \cap
V^{\otimes|\lambda|}) = V_{\lambda} \cap
(\cap_{I_1,\dots,I_k}\Ker\,(\phi_{I_1,\dots,I_k}:\,V^{\otimes|\lambda|}\rightarrow
V^{\otimes|\lambda|-2k})),$$ where $I_1,\dots,I_k$ run over all sets
of $k$ distinct pairs of indices $1,\dots,|\lambda|$ and $\phi_{I_1,\dots,I_k} = \phi_{I_1} \circ \dots \circ \phi_{I_k}$. 

Next, let $\mathcal A_\gg\subset \End_{\gg}(T(V))$ be the graded
subalgebra of $\End_{\gg}(T(V))$ generated by 
$\bigoplus_{m\geq 0}\CC[S_m]$ and the contractions
$\varphi_{i,j}$. We define a $\ZZ_{\geq 0^-}$ grading  $\mathcal
A_{\gg}=\bigoplus_{q\geq 0} (\mathcal A_{\gg})_q$ by setting 
$$(\mathcal A_{\gg})_q:=\bigoplus_{m\geq 0}{\mathrm{Hom}}_\gg(T^m,T^{m-2q})\cap\mathcal A_{\gg}.$$ 
If we set $T^{\leq r}(V):=\bigoplus_{m\leq r} T^m$ and denote by
$\mathcal A_{\gg}^{(r)}$ the intersection of  
$\mathcal A_{\gg}$ with $\End_{\gg}(T^{\leq r}(V))$, then 
$\mathcal A_{\gg}=\displaystyle\lim_{\rightarrow}\mathcal
A_{\gg}^{(r)}$. It is clear that the algebra 
$\mathcal A_\gg$ can depend only on the symmetry type of the form on
$V$ but not on $V$ and the form itself. 
This justifies the notations $\mathcal A_{\oo}$ and $ \mathcal A_{\sp}$.

\begin{prop}\label{Koszul3} \cite{DPS}
\begin{enumerate}[a)]
 \item $\mathcal A_{\oo}^{(r)}\simeq \mathcal A_{\sp}^{(r)}$ for each $r\geq 0$, and  $\mathcal A_{\oo}\simeq \mathcal A_{\sp}$.
  \item If $V$ is countable dimensional, then $(\mathcal A_{\oo})_q = \bigoplus_{m\geq 0}{\mathrm{Hom}}_{\oo(V)}(T^m, T^{m-2q}),\;(\mathcal A_{\sp})_q = \bigoplus_{m\geq 0}{\mathrm{Hom}}_{\sp(V)}(T^m, T^{m-2q})$.
  \item $\mathcal A_{\oo}^{(r)}\simeq \mathcal A_{\sp}^{(r)}$ is a Koszul ring for any $r\geq0$.
\end{enumerate}
\end{prop}

In each of the three cases $\gg = \sl(\infty), \oo(\infty), \sp(\infty)$ we call the 
modules $V_{\lambda,\mu}$, respectively $V_{\lambda,\gg},$ the \textit{simple tensor modules} of $\gg$.

\section{The category $\TT_{\gg}$}\label{sect:5}

\subsection{The countable-dimensional case}\label{sect:5.1} 
In this subsection we assume that $\gg=\sl(V,V_*),\oo(V)$ or $\sp(V)$ for a countable-dimensional space $V$. The category  $\TT_{\gg}$ has been studied in \cite{DPS}, and here we review some key results.

Denote by $\tilde{G}$ the group of automorphisms of $V$ under which $V_*$ is stable for $\gg = \sl(V,V_*)$, and the group of automorphisms of $V$ which keep fixed the form on $V$ which defines $\gg$. The group $\tilde{G}$ is a subgroup of $\mathrm{Aut}\gg$ and therefore acts naturally on isomorphism classes of $\gg-$modules: to each $\gg-$module $M$ one assigns the twisted $\gg-$module $M_{\tilde{g}}$ for $\tilde{g} \in \tilde{G}$. A $\gg-$module $M$ is \textit{$\tilde{G}-$invariant} if $M \simeq M_{\tilde{g}}$ for all $\tilde{g} \in \tilde{G}$. 

Furthermore, define a $\gg-$module $M$ to be an \textit{absolute weight module} if the decomposition (5) holds for any local Cartan subalgebra of $\gg$, i.e. if $M$ is a weight module for any local Cartan subalgebra $\hh$ of $\gg$. In \cite{DPS} we have given five equivalent characterizations of the objects of $\TT_{\gg}$.

\begin{theo}\label{th:6.1}\cite{DPS} The following conditions on a $\gg-$module $M$ of finite length are equivalent:
\begin{enumerate}[i)]
 \item $M$ is an object of $\TT_{\gg}$;
 \item $M$ is a weight module for some local Cartan subalgebra $\hh \subset \gg$ and $M$ is $\tilde{G}-$invariant;
 \item $M$ is a subquotient of $T(V\oplus V_*)$ for $\gg = \sl(V,V_*)$ (respectively, of $T(V)$ for $\gg = \oo(V), \sp(V)$);
 \item $M$ is a submodule of $T(V\oplus V_*)$ for $\gg = \sl(V,V_*)$ (respectively, of $T(V)$ for $\gg = \oo(V), \sp(V)$);
 \item $M$ is an absolute weight module.
\end{enumerate}
\end{theo}

Furthermore, the following two theorems are crucial for understanding the structure of $\TT_\gg$.

\begin{theo}\label{kos1} \cite{ps} \cite{DPS} 
   The simple objects in the categories $\widetilde{\mathrm{Tens}}_{\gg}$ and $\TT_{\gg}$ coincide and are all of the form $V_{\lambda,\mu}$ for $\gg = \sl(V,V_*)$, 
or respectively $V_{\lambda,\gg}$ for $\gg=\oo(V), \sp(V)$.
\end{theo}

\begin{theo}\label{kos2}\cite{DPS} a) $\TT_{\gg}$ has enough injectives. If $\gg=\sl(V,V_*)$, then $V_\lambda\otimes (V_*)_\mu$ is an injective hull of
$V_{\lambda,\mu}$.  If $\gg=\oo(V)$ or $\sp(V)$, then $V_\lambda$ is an injective hull of $V_{\lambda,\gg}$.

b) $\TT_{\gg}$ is anti-equivalent to the category of locally unitary finite-dimensional $\mathcal{A}_\gg-$modules.
\end{theo}

Theorem \ref{kos2} means that the category $\TT_\gg$ is ``Koszul'' in the sense that it is anti-equivalent to a module category over the infinite-dimensional 
Koszul algebra $\mathcal{A}_\gg$. 

\begin{corollary}
$\TT_{\oo(\infty)}$ and  $\TT_{\sp(\infty)}$ are equivalent abelian categories.
\end{corollary}

In fact, the stronger result that $\TT_{\oo(\infty)}$ and  $\TT_{\sp(\infty)}$ are equivalent as monoidal categories also holds, see \cite{SS} and \cite{S}.

\vspace{0.5cm}

\subsection{The general case}\label{sect:5.2}
In this subsection we prove the following result.
\begin{theo}\label{equivmain} 
Let $\gg=\sl(V,W)$, $\oo(V), \sp(V)$. Then, as a monoidal category, $\TT_\gg$ is equivalent to $\TT_{\sl(\infty)}$ or  $\TT_{\oo(\infty)}$.
\end{theo}

The proof of Theorem \ref{equivmain} is accomplished by proving several lemmas and corollaries. 

\begin{lemma}\label{inv1} 
a) Let $\gg=\sl(V,W)$ and $C_{m,n}:=\Hom_\gg(T^{m,n},\CC)$. If $m\neq n$, then $C_{m,n}=0$, and if $m=n$, then
$C_{m,m}$ is spanned by $\tau_\pi$ for all $\pi\in S_m$, where
$$\tau_\pi(v_1\otimes\dots\otimes v_m\otimes w_1\otimes \dots\otimes w_m)=\prod_{i=1}^m\,\tr(v_i\otimes w_{\pi(i)}).$$

b) Let  $\gg=\oo(V)$ or $\sp(V)$. Then $\Hom_\gg(T^{2m+1},\CC)=0$  and $\Hom_\gg(T^{2m},\CC)$ is spanned by
$\sigma_\pi$  for all $\pi\in S_m$, where
$$\sigma_\pi(v_1\otimes\dots\otimes v_{2m})=\prod_{i=1}^m(v_i,v_{m+\pi(i)}).$$
\end{lemma}

\begin{proof} In the finite-dimensional case the same statement is the fundamental theorem of invariant theory. Since $T^{m,n}$ for $\gg = \sl(V,W)$ (respectively, $T^m$ for $\gg = \oo(V),\sp(V)$) is a direct limit of finite-dimensional representations of the same type, the statement follows from the fundamental theorem of invariant theory. 
\end{proof}

Let $L$ be a $\gg$-module and let $\gg'$ denote a subalgebra of $\gg$
of the form $\sl(V', W')$ (respectively, $\oo(V'),\sp(V')$) 
for some non-degenerate pair 
$V'\subset V, W'\subset W$ (respectively, non-degenerate subspace $V' \subset V$). 
Let $(V_f',W_f')$ be a finite-dimensional non-degenerate pair satisfying
$V_f'\subset V', W_f'\subset W'$ (respectively, $V_f'\subset V'$) and
let 
$\kk'=\sl((W_f')^\perp,(V_f')^\perp)\subset \gg$ (respectively $\kk'=\oo((V_f')^\perp), \sp(V_f')^\perp)$). 
Then $L^{\kk'}$ is an $\sl(W_f',V_f')-$module (respectively, an
$\oo(V_f')-$ or $\sp(V_f')-$module), 
and moreover if we let $\kk'$ vary, the corresponding
$\sl(V_f',W_f')-$modules 
(respectively, $\oo(V_f')-$ or $\sp(V_f')-$modules) form a directed
system whose direct limit 
$$\Gamma^{ann}_{\gg'}(L) = \varinjlim L^{\kk'}$$ is a $\gg'-$module. Note
that $\Gamma_{\gg'}^{ann}(L)$ may simply be defined as the union
$\bigcup_{\kk'}\,L^{\kk'}$ of subspaces $L^{\kk'} \subset L$. It is easy to
check that 
$\Gamma^{ann}_{\gg'}$ is a well-defined functor from the category
$\gg-$mod to its subcategory of $\gg'$-mod consisting of modules
satisfying the large annihilator condition. In particular,
$\Gamma_\gg^{ann}$ is a well-defined functor from $\gg-$mod to the category
of $\gg-$modules satisfying the large annihilator condition, and the
restriction of 
$\Gamma^{ann}_\gg$ to $\mathbb T_\gg$ is the identity functor.

In the case when $\gg'$ is finite dimensional the functor $\Gamma^{ann}_{\gg'}$ and its right derived functors are studied in detail in \cite{SSW}. 

\begin{lemma}\label{inv2} 

a) Let $\gg=\sl(V,W)$, then
$$\Gamma_\gg^{ann}((T^{m,n})^*) \simeq \bigoplus_{k\geq0}b_kT^{n-k,m-k}$$
where $b_k=\binom{m}{k}\binom{n}{k}k!$.

b)  Let $\gg=\oo(V)$ or $\sp(V)$, then
$$\Gamma_\gg^{ann}((T^m)^*) \simeq \bigoplus_{k\geq0}c_kT^{m-2k}$$
where $c_k=\binom{m}{2k}k!$.
\end{lemma} 

\begin{proof} We prove a) and leave b) to the reader.
Choose a finite-dimensional non-degenerate pair $V_f \subset V, W_f \subset W$, and let $\kk=\sl(W_f^\perp,V_f^\perp)$. There is an isomorphism of $\kk$-modules 

\begin{equation}\label{eq:24.10}
 (T^{m,n})^*=(V^{\otimes m}\otimes W^{\otimes n})^*\simeq \bigoplus_{k\geq0,l\geq0}d_{k,l}(W_f^{\otimes m-k}\otimes V_f^{\otimes n-l})\otimes ((V_f^\perp)^{\otimes k}\otimes (W_f^\perp)^{\otimes l})^*
\end{equation}
where $d_{k,l}=\binom{m}{k}\binom{n}{l}$.

Using (\ref{eq:24.10}) and Lemma \ref{inv1} a) applied to $\kk$ in place of $\gg$, we compute that
$$\left((T^{m,n})^*\right)^\kk\simeq  \bigoplus_{k\geq0}b_{k}(W_f^{\otimes m-k}\otimes V_f^{\otimes n-k}).$$
Now the statement follows by taking the direct limit of $\kk-$invariants over all non-degenerate finite-dimensional pairs $V_f\subset V, W_f\subset W$.
\end{proof}

\begin{corollary}\label{injective} $T^{m,n}$ is an injective object of $\TT_{\sl(V,W)}$, and $T^m$ is an injective object of $\TT_\gg$ for $\gg = \oo(V),\sp(V)$.
\end{corollary}

\begin{proof} We consider only the case $\gg = \sl(V,W)$. Recall (Theorem \ref{int1}) that if $M$ is an integrable module such that $M^*$ is integrable, then $M^*$ is injective in $\Int_\gg$. In particular,  $(T^{m,n})^*$ is injective in $\Int_\gg$. Next, note that $\Gamma_\gg^{ann}$ is right adjoint to the inclusion functor $\TT_\gg \rightsquigarrow \Int_\gg$, i.e. for any $L\in\TT_\gg$ and any $Y \in \Int_\gg$, we have $$\Hom_{\gg} (L,Y)=\Hom_{\gg} (L,\Gamma_\gg^{ann}(Y)).$$ Hence, $\Gamma_\gg^{ann}$ transforms injectives in $\Int_\gg$ to injectives in $\TT_\gg$. This implies that $\Gamma_\gg^{ann}((T^{n,m})^*)$ is injective  in $\TT_\gg$. By Lemma \ref{inv2},  $T^{m,n}$ is a direct summand in $\Gamma_\gg^{ann}((T^{n,m})^*)$, and the statement follows.
\end{proof}

Next we impose the condition that our fixed subalgebra $\gg' \subset
\gg$ is countable dimensional. In the rest of the paper we set $\gg_c
:= \gg'$. More precisely, we choose strictly increasing chains of
finite-dimensional subspaces
$$V_1\subset V_2\subset\dots V_i \subset V_{i+1} \subset \dots,
\;\;W_1\subset W_2\subset\dots W_i \subset W_{i+1} \subset \dots$$ and
set $\gg_c = \sl(V_c,W_c)$ where $V_c := \varinjlim V_i, W_c :=\varinjlim W_i$. 
It is clear that 
$V_c \times W_c \rightarrow \CC$ is a countable-dimensional linear
system, hence $\gg_c \simeq \sl(\infty)$. 
If $\gg = \oo(V),\sp(V)$ choose a strictly increasing chain of
non-degenerate finite-dimensional subspaces 
$V_1 \subset V_2 \dots \subset V_i\subset V_{i+1} \subset \dots $ and
set 
$V_c:= \varinjlim \, V_i$, $\gg_c = \oo(V_c), \sp(V_c)$. 

By $\Phi$ we denote the restriction of $\Gamma^{ann}_{\gg_c}$ to
$\TT_\gg$. Note that 
for any $L \in \TT_\gg,\,\Phi(L)$ is a $\gg_c-$submodule of $L$.

\begin{lemma}\label{propfunctor} Let $L,L'\in\TT_\gg$.

a) $\Phi(L)$ generates $L$.

b) The homomorphism $\Phi(L,L'):\Hom_\gg(L,L')\to\Hom_{\gg_c}(\Phi(L),\Phi(L'))$
is injective.

\end{lemma}

\begin{proof} Again we consider only the case $\gg=\sl(V,W)$ 
since the other cases are similar. Let $SL(V,W)$ denote the direct limit group
$\varinjlim\,SL(V_f,W_f)$ for all non-degenerate finite-dimensional
pairs $V_f \subset V, W_f \subset W$, where $SL(V_f,W_f) \simeq SL(\dim\,V_f)$.

a) Since $L$ has finite length and satisfies the large annihilator
condition, there is a finite-dimensional 
non-degenerate pair $V_f\subset V,W_f\subset W$ and a
finite-dimensional $\gl(V_f,W_f)$-submodule 
$L_f \subset L$ annihilated by $\sl((W_f)^\perp,(V_f)^\perp)$ such
that $L$ is generated by $L_f$ over $\gg$. 
Choose $i$ so that $\dim V_f<\dim V_i$. Then there exists $g\in SL(V,W)$ such that
$g(V_f)\subset V_i,g(W_f)\subset W_i$. Note that $g=\exp x$ for some
$x\in \sl(V,W)$.  
By the integrability of $L$ as a $\gg-$module,
the action of $g$ is well defined on $L$, and $g(L_f)$ also generates
$L$ over $\gg$. 
On the other hand, by construction $g(L_f)$ is annihilated by $g\sl((W_f)^\perp,(V_f)^\perp)g^{-1}$.
Observe that
$$\sl((W_i)^\perp,(V_i)^\perp)\subset\sl(g(W_f)^\perp,g(V_f)^\perp)=g\sl((W_f)^\perp,(V_f)^\perp)g^{-1}.$$ 
Hence $g(L_f)\subset\Phi(L)$. The statement follows.

b) follows immediately from a).
\end{proof}

\begin{lemma}\label{tensorfunctor} 

a) $\Phi(T^{m,n})=V_c^{\otimes m}\otimes W_c^{\otimes n}$ for $\gg =
\sl(V,W)$, and 
$\Phi(T^m)=V_c^{\otimes m}$ for $\gg=\oo(V)$, $\sp(V)$;

b) The homomorphisms 
$$\Phi(T^{m,n},T^{k,l}):\Hom_\gg(T^{m,n},T^{k,l})\to \Hom_{\gg_c}(V_c^{\otimes m}\otimes W_c^{\otimes n},V_c^{\otimes k}\otimes W_c^{\otimes l})$$ 
for $\gg=\sl(V,W)$, and
$$\Phi(T^m,T^k):\Hom_\gg(T^m,T^k)\to \Hom_{\gg_c}(V_c^{\otimes k},V_c^{\otimes k})$$
for  $\gg=\oo(V)$ or $\sp(V)$, are isomorphisms.

c) Let $X\subset \bigoplus_i V_c^{\otimes m_i}\otimes W_c^{\otimes  n_i}$, 
(respectively, $X \subset \bigoplus_i V_c^{m_i}$ for $\gg = \oo(V),\sp(V)$) be a $\gg_c$-submodule. Then $\Phi(U(\gg)\cdot X)=X$.

d) If  $X\subset V_c^{\otimes m}\otimes W_c^{\otimes n}$ (respectively, $X \subset \bigoplus_i V_c^{m_i}$ for $\gg = \oo(V),\,\sp(V)$) is a simple submodule, then $U(\gg)\cdot X$ is a simple $\gg-$module.
\end{lemma}
\begin{proof} a) follows easily from the observation that
$$(T^{m,n})^{\kk}=V_i^{\otimes m}\otimes W_i^{\otimes n}$$ for any finite corank subalgebra $\kk=\sl(W_i^\perp,V_i^\perp)$. This observation is a straightforward consequence of Lemma \ref{inv}.

To prove b), note that the injectivity of the homomorphisms $\Phi(T^{m,n},T^{k,l})$ follows from a) and Lemma \ref{propfunctor} b). To prove surjectivity, we observe that
$\Hom_{\gg_c}(V_c^{\otimes m}\otimes W_c^{\otimes n},V_c^{\otimes k}\otimes W_c^{\otimes l})$ is generated by permutations and contractions according to Proposition \ref{Koszul3} b). Both are defined in $\Hom_\gg(T^{m,n},T^{k,l})$ by the same formulae. Therefore the homomorphisms $\Phi(T^{m,n}, T^{k,l})$ are surjective.  

We now prove c). Note that $X=\Ker \alpha$ for some 
$\alpha\in\Hom_{\gg_c}( \bigoplus_i V_c^{\otimes m_i}\otimes W_c^{\otimes n_i}, \bigoplus_j V_c^{\otimes m_j}\otimes W_c^{\otimes n_j})$.
Using b) we have $U(\gg)\cdot X\subset\Ker \Phi^{-1}(\alpha)$. Hence, $\Phi(U(\gg)\cdot X)\subset\Ker\alpha=X$. Since the inclusion $X\subset\Phi(U(\gg)\cdot X)$ is obvious,
the statement follows.

To prove d), suppose $U(\gg)\cdot X$ is not simple, i.e. there is an exact sequence
$$0\to L\to U(\gg)\cdot X\to L'\to 0$$
for some non-zero $L,L'$. By the exactness of $\Phi$ and by c), we have an exact sequence
$$0\to \Phi(L)\to X\to \Phi(L')\to 0.$$
By Lemma \ref{propfunctor} a), $\Phi(L)$ and $\Phi(L')$ are both non-zero. This contradicts the assumption that $X$ is simple.
\end{proof}

\vspace{0.33cm}
\begin{lemma}\label{tensor} For $\gg = \sl(V,W)$ (respectively, for  $\gg = \oo(V), \sp(V)$) 
any simple object in the category $\TT_\gg$ is isomorphic to a
submodule in $T^{m,n}$ for suitable $m$ and $n$ 
(respectively, in $T^m$ for a suitable $m$).
\end{lemma}
\begin{proof} We assume that $\gg = \sl(V,W)$ and leave the other cases to the reader. 
Let $L$ be a simple module in $\TT_\gg$. By Lemma \ref{propfunctor}
a), 
$\Phi(L)\neq 0$. Let $L_i = L^{\sl(W_i^{\perp},V_i^{\perp})}\neq0$ for
some $i$, and let $L'\subset L_i$ be a simple $\sl(V_i,W_i)$-submodule. Consider the $\ZZ$-grading
$\gg=\gg^{-1}\oplus\gg^0\oplus\gg^1$ where 
$\gg^0=\gl(V_i,W_i)\oplus \sl(W_i^\perp,V_i^\perp)$, $\gg^1=V_i\otimes V_i^\perp$,  
$\gg^{-1}=W_i^\perp\otimes W_i$. There exists a finite-dimensional
subspace $W'\subset V_i^\perp$, such that 
$S(V_i\otimes W')$ generates $S(\gg^1)$ as a module over
$\sl(W_i^\perp,V_i^\perp)$. By the integrability of $L$, 
$(V_i\otimes W')^q\cdot L'=0$ for sufficiently large $q\in \ZZ_{\geq  0}$, and thus 
$(\gg^1)^q\cdot L'=0$. Hence, there is a non-zero vector $l\in L_i
\subset L$ annihilated by $\gg^1$, and consequently 
there is a simple $\gg^0$-submodule $L^{''} \subset L$ annihilated by
$\gg^1$. 
Therefore $L$ is isomorphic to a quotient of the parabolically induced module 
$U(\gg)\otimes_{U(\gg^0\oplus\gg^1)}L^{''}$. The latter module is a direct limit of 
parabolically induced modules for finite-dimensional subalgebras of
$\gg$. 
Hence it has a unique integrable quotient, and this quotient is
isomorphic to $L$. 
On the other hand, $L^{''}$ is a simple $\gg_0$-submodule of $T^{m,n}$
for some $m$ and $n$. 
Thus, by Frobenius reciprocity, a quotient of $U(\gg)\otimes_{U(\gg^0\oplus \gg^1)}L"$ is isomorphic to a submodule of $T^{m,n}$. Since $T^{m,n}$ is integrable, this quotient is isomorphic to $L$. 
\end{proof}

\vspace{0.5cm}
\begin{corollary}\label{corr:1}
 \begin{enumerate}[a)]
  \item If $\gg = \sl(V,W)$, then $\mathcal A_\sl =    \bigoplus_{m,n,q}\,\Hom\,_{\gg}(T^{m,n},T^{m-q,n-q})$. 
If $\gg = \oo(V), \sp(V)$, then 
$\mathcal A_\gg = \bigoplus_{m,q}\,\Hom\,(T^m, T^{m-2q})$. Furthermore,
$$\mathcal A_{\sl}=\varinjlim\operatorname{End}_{\gg}(\bigoplus_{m+n\leq r}T^{m,n}),$$
and for $\gg = \oo(V), \sp(V)$
$$\mathcal A_{\oo}=\varinjlim\operatorname{End}_{\gg}(\bigoplus_{m\leq r}T^{m}).$$ 
  \item Up to isomorphism, the objects of $\TT_\gg$ are precisely all
    finite length submodules of $T(V,W)^{\oplus k}$ 
for $\gg = \sl(V,W)$, and of $T(V)^{\oplus k}$ 
for $\gg = \oo(V), \sp(V)$. Equivalently, up to isomorphism, the
objects of 
$\TT_\gg$ are the finite length subquotients of $T(V,W)^{\oplus k}$ for $\gg = \sl(V,W)$, 
and of $T(V)^{\oplus k}$ for $\gg = \oo(V), \sp(V)$.
 \end{enumerate}
\end{corollary}
\begin{proof} Claim a) is a consequence of Lemma \ref{tensorfunctor}. 
Claim b) follows from Lemma \ref{tensor} and Corollary \ref{injective}.
\end{proof}

\begin{lemma}\label{functor} For any $L\in \TT_\gg$, $\Phi(L)\in  \TT_{\gg_c}$. Moreover,
the functor $\Phi: \TT_\gg\to\TT_{\gg_c}$ is fully faithful and essentially surjective.
\end{lemma}
\begin{proof} By Corollary \ref{corr:1} b), $L$ is isomorphic to
is a submodule in a direct sum of finitely many copies of
$T(V,W)$. Then $\Phi(L)$ is isomorphic to a submodule in  a direct sum of finitely
many copies of $T(V_c,W_c)$. That implies the first assertion. The
fact that $\Phi$ is faithful follows from Lemma \ref{propfunctor}
b). 

To prove that $\Phi$ is full, consider $L,L'\in  \TT_\gg$ and let
$I(L),I(L')$ denote respective injective hulls in $\TT_\gg$. Then 
$$\operatorname{Hom}_\gg(L,L')\subset\operatorname{Hom}_\gg(I(L),I(L'))$$ and 
$$\operatorname{Hom}_{\gg_c}(\Phi(L),\Phi(L'))\subset\operatorname{Hom}_{\gg_c}(\Phi(I(L)),\Phi(I(L'))).$$ By
Corollary \ref{corr:1} a), the homomorphism
$$\Phi(I(L),I(L')):\operatorname{Hom}_\gg(I(L),I(L'))\to\operatorname{Hom}_{\gg_c}(\Phi(I(L)),\Phi(I(L')))$$ 
is surjective. Therefore for any $\varphi\in \operatorname{Hom}_{\gg_c}(\Phi(L),\Phi(L'))$ there exists 
$\psi\in \operatorname{Hom}_\gg(I(L),I(L'))$ such that $\psi(\Phi(L))\subset \Phi(L')$. By Lemma \ref{propfunctor}  
$\Phi(L)$ and $\Phi(L')$ generate respectively
$L$ and $L'$. Hence $\psi(L)\subset L'$. Thus, we obtain that
the homomorphism 
$$\Phi(L,L'):\operatorname{Hom}_\gg(L,L')\to\operatorname{Hom}_{\gg_c}(\Phi(L),\Phi(L'))$$ is also surjective. 

To prove that $\Phi$ is essentially surjective, we use again
Corollary \ref{corr:1} b). We note that any $L\in  \TT_\gg$ is
isomorphic to the kernel of $\varphi\in\operatorname{Hom}(T(V,W)^{\oplus k},(T(V,W)^{\oplus l})$
for some $k$ and $l$ and then apply Corollary \ref{corr:1} a).
\end{proof}

Observe that Lemma \ref{functor} implies that $$\Phi:\TT_\gg \rightarrow \TT_{\gg_c}$$ an equivalence of the abelian
categories $\TT_\gg$ and $\TT_{\gg_c}$.
To prove Theorem \ref{equivmain} it remains 
to check that $\Phi$ is an equivalence of monoidal categories. We therefore prove the following.

\begin{lemma}\label{tensphi} If $L,N\in\TT_{\gg}$, then $\Phi(L\otimes N)\simeq \Phi(L)\otimes \Phi(N)$.
\end{lemma}

\begin{proof} We just consider the case $\sl(V,W)$ as the  orthogonal and symplectic cases are very similar.
Let $\kk=\sl(W_f^\perp,V_f^\perp)$ for some finite-dimensional
non-degenerate pair 
$V_f \subset V ,W_f \subset W$. We claim that 
$$(L\otimes N)^{\kk}=L^\kk\otimes N^\kk.$$
Indeed, using Lemma \ref{inv} one can easily show that 
$$(T^{m,n})^{\kk}=V_f^{\otimes m}\otimes W_f^{\otimes n},$$
which implies the statement in the case when $L$ and $N$ are injective.
For arbitrary $L$ and $N$ consider embeddings $L\hookrightarrow I$ and $N\hookrightarrow J$
for some injective $I,J\in\TT_\gg$. Then
$$(L\otimes N)^\kk=(L\otimes N)\cap (I\otimes J)^\kk=(L\otimes N)\cap (I^\kk\otimes J^\kk)=L^\kk\otimes N^\kk.$$

Now we set $\kk=\sl(W_i^{\perp},V_i^{\perp})$ and finish the proof by passing to the direct limit.
\end{proof}

\vspace{0.5cm}

The proof of Theorem \ref{equivmain} is complete. $\square$

\section{Mackey Lie algebras}\label{sect:6}
Let $V \times W \rightarrow \CC$ be a linear system.
Then each of $V$ and $W$ can be considered as subspace of 
the dual of the other:
\begin{eqnarray}
 V \subset W^*, \;\;W \subset V^*. \nonumber
\end{eqnarray}
Let $\End_W(V)$ denote the algebra of endomorphisms $\phi: V \rightarrow V$ such that $\phi^*(W) \subset W$ where $\phi^*: V^* \rightarrow V^*$ is the dual 
endomorphism. Clearly, there is a canonical anti-isomorphism of algebras
\begin{equation}\label{rel:canonic}
 \End_W(V) \stackrel{\sim}{\rightarrow} \End_V(W), \,\,\, \phi \longmapsto \phi^*_{|W}.
\end{equation}
We call the Lie algebra associated with the associative algebra $\End_W(V)$ (or equivalently  $\End_V(W)$) a \textit{Mackey Lie algebra} and denote it by 
$\gl^M(V,W)$.

Note that if $V,\,W$ is a linear system, then for any subspaces $W' \subset V^*$ with $W \subset W',$ and $V' \subset W^*$ with $V \subset V'$, the pairs $V,\,W'$ 
and $V',\,W$ are linear systems. In particular $V,\,V^*$ is a linear system and $W^*,\,W$ is a linear system. Clearly, $\gl^M(V,V^*)$ coincides with the Lie 
algebra of all endomorphisms of $V$ (respectively, $\gl^M(W^*,W)$ is the Lie algebra of all endomorphisms of $W$). Hence 
$\gl^M(V,W) \subset \gl^M(V,V^*)$, $\gl^M(V,W) \subset \gl^M(W^*,W)$.
If $V$ and $W = V_*$ are countable dimensional, the Lie algebra $\gl^M(V,V_*)$ is identified 
with the Lie algebra of all matrices $X=(x_{ij})_{i\geq 1, j\geq 1}$ such that each row and each column of $X$ have finitely many non-zero entries. 
The Mackey Lie algebra  $\gl^M(V,V^*)$
(for a countable dimensional space $V$) is identified with the Lie algebra of all matrices $X=(x_{ij})_{i\geq 1, j\geq 1}$ each column of which has finitely 
many non-zero entries. Alternatively, 
if a basis of $V$ as above is enumerated by $\ZZ$ (i.e we consider a basis $\lbrace v_j\rbrace_{j \in \ZZ}$ 
such that $V_* = \span \lbrace v^*_j \rbrace_{j \in \ZZ}$ 
where $v^*_j(v_i) = 0$ for $j \neq i$, $v_j^*(v_j)=1$), then  $\gl^M(V,V_*)$ is identified with the Lie algebra of all matrices $(x_{ij})_{i,j \in \ZZ}$ whose rows and 
columns have finitely many non-zero entries, and  $\gl^M(V,V^*)$ is identified with the Lie algebra of all matrices $(x_{ij})_{i,j\in \ZZ}$ whose columns have finitely 
many non-zero entries.

Obviously $V$ and $W$ are  $\gl^M(V,W)$ - modules. Moreover, $V$ and $W$ are not isomorphic as  $\gl^M(V,W)$ - modules.

It is easy to see that $\gl(V,W) = V \otimes W$ is the subalgebra of $\gl^M(V,W)$ consisting of operators with finite-diemnsional images in both $V$ and $W$, and that it is an ideal in  $\gl^M(V,W)$. Furthemore, the Lie algebra  $\gl^M(V,W)$ has a 1-dimensional center consisting of 
the scalar operators $\CC \id$.

We now introduce the orthogonal and symplectic Mackey Lie algebras. Let $V$ be a vector space endowed with a non-degenerate symmetric 
(respectively, antisymmetric) form, then $\oo^M(V)$ (respectively, $\sp^M(V)$) is the Lie algebra
\begin{equation}\label{rel:Liealg}
 \lbrace X \in \End(V)|\,(X\cdot v,w)+(v,X\cdot w)=0 \,\,\, \forall \, v,w \in V\rbrace.
\end{equation}
If $V$ is countable dimensional, there always is a basis $\lbrace v_i,w_j\rbrace_{i,j \in \ZZ}$ of $V$ such that $\span\lbrace v_i\rbrace_{i \in \ZZ}$ and $\span\lbrace w_j\rbrace_{j \in \ZZ}$ are isotropic spaces and $(v_i, w_j) = 0$ for $i \neq j$, $(v_i, w_i)=1$. The corresponding matrix form of  $\oo^M(V)$ consists of all matrices
\begin{equation}\label{matrix}
\left(
\begin{array}{c|c}
 a_{ij} & b_{kl} \\ \hline
 c_{rs} & -a_{ji} 
\end{array}
\right) 
\end{equation}
each row and column of which are finite and in addition $b_{kl}=-b_{lk},\,c_{rs}=-c_{sr}$ where $i,j,k,l,r,s \in \ZZ$. The matrix form for $\sp^M(V)$ is similar: 
here $b_{kl} = b_{lk},\,c_{rs}=c_{sr}$.

It is clear that $\oo(V)\subset \oo^M(V)$ and
$\sp(V)\subset\sp^M(V)$:
\begin{equation}\label{eq:v_wedge_x}
 (v\wedge w)\cdot x = (v,x)w - (x,w)v \;\;\mathrm{for}\; v\wedge w \in \Lambda^2V = \oo(V),\;x\in V
\end{equation}
and
\begin{equation}
 (vw)\cdot x = (v,x)w - (x,w)v \;\;\mathrm{for}\; vw \in S^2V = \sp(V),\;x \in V.
\end{equation}

Moreover,  $\oo(V)$ is an ideal in $\oo^M(V)$ and $\sp(V)$ is an ideal in $\sp^M(V)$, since both $\Lambda^2 V$ and $S^2V$ consist of the respective operators with finite-dimensional image in $V$.

In this way, we have the following exact sequences of Lie algebras
\begin{equation}\label{seq:Lalg1}
 0 \rightarrow \gl(V,W) \rightarrow \gl^M(V,W) \rightarrow \gl^M(V,W) /\gl(V,W)  \rightarrow 0,  \\
\end{equation}
\begin{equation}\label{seq:Lalg2}
0 \rightarrow \oo(V) \rightarrow \oo^M(V) \rightarrow \oo^M(V) / \oo(V) \rightarrow 0,  \\ 
\end{equation}
\begin{equation}\label{seq:Lalg3}
0 \rightarrow \sp(V) \rightarrow \sp^M(V) \rightarrow \sp^M(V) / \sp(V) \rightarrow 0. 
\end{equation}

\begin{lemma}\label{lem_nz} $\sl(V,W)$ (respectively, $\oo(V), \sp(V)$) is the unique simple ideal in  $\gl^M(V,W)$ (respectively, $\oo^M(V), \sp^M(V)$). 
\end{lemma}

\begin{proof}
 We will prove that if $I\neq\CC\id$ is a non-zero ideal in $\gl^M(V,W)$, then $I$ contains  $\sl(V,W)$. Indeed, assume that $X\in I$ and $X\neq c\id$.
Then one can find $v\in V$ and $w\in W$ such that $X\cdot v$ is not proportional to $v$ and $X^*\cdot w$ is not proportional to $w$.
Hence, $Z=[X,v \otimes w]=(X\cdot v)\otimes w- v \otimes (X\cdot w)\in \gl(V,W) \cap I$ and $Z\neq 0$. Since $\sl(V,W)$ is the 
unique simple ideal in $\gl(V,W)$ and $\gl(V,W) \cap I\neq 0,$  we conclude that $\sl(V,W)\subset I$.

The two other cases are similar and we leave them to the reader.
\end{proof}

\begin{corollary}\label{corr:14_10}
a) Two Lie algebras $\gl^M(V,W)$ and $\gl^M(V',W')$ are isomorphic if and only if the linear systems $V\times W\rightarrow \CC$ and 
$V' \times W' \rightarrow \CC$ are isomorphic.

b) Two Lie algebras $\oo^M(V)$ and $\oo^M(V')$ (respectively, $\sp^M(V)$ and $\sp^M(V')$) are isomorphic if and only if there is an isomorphism of vector spaces $V\simeq V'$ transferring the form defining $\oo^M(V)$ (respectively $\sp^M(V)$) into the form defining $\oo^M(V')$ (respectively, $\sp^M(V')$).
\end{corollary}

\begin{proof}
 The statement follows from Proposition \ref{prop:1.29.9} and Lemma \ref{lem_nz}.
\end{proof}

\vspace{0.33cm}

The following is our main result about the structure of Mackey Lie algebras.

\begin{theo} \label{maintheo} Let $V$ be a countable-dimensional vector space. 

a) $\gl(V,V_*) \oplus \CC \id$ is an ideal in $\gl^M(V,V_*)$ and the 
quotient $$\gl^M(V,V_*)/\left(\gl(V,V_*)\oplus \CC \id\right)$$ 
is a simple Lie algebra.

b)  $\gl(V, V^*) \oplus \CC \id$ is an ideal in $\End(V)$ and the quotient
$\End(V)/\left(\gl(V, V^*) \oplus \CC \id\right)$ 
is a simple Lie algebra. 

c) If $V$ is equipped with a non-degenerate symmetric (respectively, antisymmetric) bilinear form, then $\oo^M(V)/\oo(V)$ (respectively $\sp^M(V)/\sp(V)$) 
is a simple Lie algebra.
\end{theo}

\begin{proof} The proof is subdivided into lemmas and corollaries. 

Note that $\gl(V,V_*)\subset\gl(V,V^*)\subset \gl^M(V,V^*)=\End(V)$.
In what follows we fix a basis $\{v_i\}_{i\geq 1}$ in $V$ and use the respective identification of $\gl(V,V_*)$, $\gl^M(V,V_*)$ and  
$\gl^M(V,V^*)=\End(V)$
with infinite matrices.
By $E_{ij}$ we denote the elementary matrix whose only non-zero entry is $1$ at position $i,j$.

\begin{lemma}\label{lem_diag} Let $\gg^M =\gl^M(V,V_*)$,  $\End(V)$. Assume that an ideal $I \subset \gg^M$ contains a diagonal matrix $D \notin \gl(V,V_*) \oplus \CC \id$. Then $I = \gg^M$.
\end{lemma}

\begin{proof}
We first assume that $D=\sum_{i\geq 1} d_i E_{ii}$ satisfies $d_i \neq d_j$ for all $i \neq j$.
Then $[D,\gg^M] = \gg^M_0$, where $\gg^M_0$ is the space of all matrices in $\gg^M$ with zeroes on the diagonal. Consequently, $\gg^M_0 \subset I$. 
Furthermore, any diagonal matrix $\sum_{i} s_i E_{ii}$ can be written as the commutator
\begin{equation}
 \left[\sum_{i \geq 1} E_{i\,i+1}, \sum_{j \geq1} t_j E_{j+1\,j}\right] \nonumber
\end{equation}
with $t_j = \sum_{i=1}^{j} s_i$.
Hence, $I=\gg^M$.

We now consider the case of an arbitrary $D \in I$. After permuting the basis elements of $V$, we can assume that 
$D = \sum_{i\geq 1} d_i E_{ii}$ with $d_{2m-1} \neq 0$ and  $d_{2m-1} \neq d_{2m}$ for all $m > 0$. Let 
\begin{equation}\label{eq:lem_diag}
 X := \sum_{m=1}^{\infty} \frac{1}{d_{2m}-d_{2m-1}}E_{2m\,2m-1}, \,\,\, Y:=\sum_{m=1}^{\infty} s_m E_{2m-1\,2m}, \nonumber
\end{equation}
where $s_m\neq \pm s_l$ for $m \neq l$. Then $[Y,[X,D]]=s_1E_{11} - s_1E_{22} + s_2E_{33} -s_2E_{44} +\dots \in I$, and we reduce this case to the previous one.
\end{proof}

\begin{lemma}\label{lem_diag2} 
Let $y = (y_{ij})\in \gl(n)$ be a non-scalar matrix. There exist $u,v,w \in \gl(n)$ such that $[u,[v,[w,y]]$ is a diagonal matrix.
\end{lemma}
\begin{proof}
 If $y$ is not diagonal, pick $i \neq j$ such that $y_{ij} \neq 0$. Set $w=E_{ii}, v=E_{jj}, u=E_{ji}$. If y is diagonal, pick $i \neq j$ such that $y_{ii} \neq y_{jj}$ and set $w=E_{ij}, v=E_{ii}, u=E_{ji}$.
\end{proof}

\begin{corollary}\label{cor_1} Let $\prod_{i} \gl(n_i)$ for $n_i \geq 2$ be a block subalgebra of $\gg^M$. 
Suppose that X $\in \left(\prod_{i} \gl(n_i)\right) \cap I$ for some ideal $I \subset \gg^M$ and that $X \notin \gl(V,V_*) \oplus \CC \mathrm{Id}$. Then $I = \gg^M$.
\end{corollary}
\begin{proof}
 Let $X = \prod_i X_i,$ where $X_i \in \gl(n_i)$. Without loss of generality we may assume that infinitely many $X_i$ are not diagonal, as otherwise $X$ is diagonal modulo $\gl(V,V_*)$ and the result follows from Lemma \ref{lem_diag}. Now pick $u_i,v_i,w_i \in \oj_i$ as in the previous lemma. Set $u = \prod_i u_i, v=\prod_i v_i, w=\prod_i w_i$. Then $Z=[u,[v,[w,X]]$ is diagonal. By normalizing $u_i$ we can ensure that $Z \notin \CC \id$. Since $Z \in I$, the statement follows from Lemma \ref{lem_diag}.
\end{proof}

\begin{lemma}\label{lem_4} 
For any $X = (x_{ij})_{i\geq1,j\geq1}\in \gl^M(V,V_*)$ there exists an increasing sequence $\lbrace i_1<i_2<\dots\rbrace$ such that 
$x_{ij}=0$ unless $i,j \in \left[i_k,i_{k+2}-1\right]$ for some $k$. 
\end{lemma}

\begin{proof} Set $i_1 =1$,
 \begin{equation}
  i_2 = \max\lbrace j\,|\, x_{1j}\neq0\; \mbox{or}\; x_{j1}\neq0\rbrace+1, \nonumber 
 \end{equation}
and construct the sequence recursively by setting 
\begin{equation}
 i_k = \max \lbrace j > i_{k-1}\,|\, x_{ij} \neq0\; \mbox{or}\; x_{ji} \neq0 \;\mbox{for some}\; i_{k-2} \leq i < i_{k-1}\rbrace +1.\nonumber
\end{equation}
\end{proof}

We are now ready to prove Theorem \ref{maintheo} a).

\begin{corollary}\label{cor_2}(Theorem \ref{maintheo} a))  
Let an ideal $I$ of $ \gl(V,V_*)$ be not contained in $\gl (V,V_*) \oplus \CC \id$. Then $I=\gl^M(V,V_*)$.
\end{corollary}

\begin{proof}
 Let $X \in I \backslash \lbrace\ojl(V,V_*) \oplus \CC \mathrm{Id}\rbrace$. Pick $i_1 < i_2 < \dots$ as in Lemma \ref{lem_4} and set 
\begin{equation}
 D = \mathrm{diag} (\underbrace{1, \dots, 1}_{i_2-1}, \underbrace{2, \dots ,2}_{i_3-i_2}, \underbrace{3, \dots, 3}_{i_4-i_3}, \dots). \nonumber
\end{equation}
 Then $X=X_{-1}+X_0+X_1$ where $[D,X_i]=iX_i$. If $X_0 \notin \ojl(V,V_*) \oplus \CC \id$ we are done by Corollary \ref{cor_1} as $X_0$ is a block matrix. Otherwise, at least one of $X_1$ and $X_{-1}$ does not lie in $\ojl(V,V_*)$.

 Assume for example that $X_1 = (x_{ij}) \notin \gl(V,V_*)$. Then  there exist infinite sequences $\lbrace i_1<i_2<\dots\rbrace$ and $\lbrace j_1<j_2<\dots\rbrace$ such that
$x_{i_sj_s}\neq 0$. Moreover, we may assume that $.\,.\,.< i_s \leq j_s < i_{s+1} \leq j_{s+1} <\dots$. Set
$Y=\sum_{s\geq 1} E_{j_si_s}$. Then $[Y,X_1]\in I$ is a block matrix and we can again use Corollary \ref{cor_1}.
\end{proof}

Next we prove Theorem \ref{maintheo}.b).

Let $I$ be an ideal in $\End(V)$. Assume that $I$ is not contained in
$\gl(V,V^*) \oplus \CC \id$. Let 
$X\in I\setminus \left\lbrace\gl(V,V^*) \oplus \CC \id\right\rbrace$ 
and let $V_X \subset V$ denote the subspace  of all $X$-finite vectors. 

Assume first that $V_X\neq V$. Then there exists $v\in V$ such that $v,X\cdot v,X^2\cdot v,\dots$ are linearly 
independent. Let $M=\span\{v,X\cdot v,X^2\cdot v,\dots\}$ and $U$ be a
subspace of $V$ such that $V=M\oplus U$. Let $\pi_M$ be the projector
on $M$ with kernel $U$. Then $Y:=X+[X,\pi_M]\in I$. A simple
calculation shows that both $U$ and $M$ are $Y$-stable and
$Y|_M=X|_M$. Let $Z\in \End(M)$ be defined by $Z(U)=0$, $Z(X^i\cdot
v)=iX^{i-1}\cdot v$ for $i\geq 0$. Then $[Z,Y]$ is a diagonal matrix
with infinitely many distinct entries. Hence $I=\End(V)$ by Lemma \ref{lem_diag}.

Now suppose that $V_X=V$. Then we have a decomposition
$V=\bigoplus_{\lambda} V_\lambda$, where 
$V_\lambda:=\bigcup_n \Ker (X-\lambda\id)^n$ are generalized
eigenspaces of $X$. First, we assume that for all 
$\lambda$ there exists $n(\lambda)$ such that $V_\lambda=\Ker
(X-\lambda\id)^{n(\lambda)}$. 
In this case $V=\bigoplus_i V_i$ is a direct sum of $X-$stable
finite-dimensional subspaces. Thus $X$ is a block matrix and by 
Corollary \ref{cor_1} we obtain $I=\End (V)$. Next, we assume that
for some $\lambda$ the sequence  $\Ker (X-\lambda\id)^n$ does not
stabilize. In this case there are linearly independent vectors
$v_1,v_2,\dots$ such that $(X-\lambda\id)\cdot v_1=0$ and
$(X-\lambda\id)\cdot v_i=v_{i-1}$ for all $i>1$. 
We repeat the argument from the previous paragraph. Set $M$ to be the span of 
$v_k$, let $V=M\oplus U$ and define $Z \in \End(M)$ by setting $Z(U)=0$, $Z(v_i)=iv_{i+1}$. Then $[Z,([X,\pi_M]+X)]\in I$  is a diagonal matrix with infinitely many distinct entries. Hence $I=\End(V)$.

To complete the proof of Theorem \ref{maintheo} it remains to prove claim c).

\begin{lemma}\label{lem_idideal}
 If $\gg^M = \oo^M(V)$ (respectively, $ \sp^M(V)$, then any non-zero proper ideal $I \subset \gg^M$ 
equals $\oo(V)$  (respectively, $\sp(V)$.
\end{lemma}

\begin{proof} As follows from (\ref{matrix}), one can define a $\mathbb Z$-grading 
 $\gg^M=\gg^M_{-1} \oplus \gg^M_0 \oplus \gg^M_1$ such that $\gg^M_0 \simeq {\ojl}^M(V,V_*)$. 
This  grading is defined by the matrix
\begin{equation}\label{grad_matrix}
D=\left(
\begin{array}{c|c}
 \frac{1}{2}\mathrm{Id} & 0 \\ \hline
 0 & -\frac{1}{2}\mathrm{Id} 
\end{array}
\right), 
\end{equation}
i.e. $[D,X] = iX$ for $X \in \gg^M_i$. Since $D\in\gg^M$, any ideal $I\subset\gg^M$ is homogeneous in this grading. Note that the ideal generated by $D$ equals the entire Lie  algebra  $\gg^M$. Hence we may assume that $D\notin I$, and thus that $I_0 := I \cap \gg^M_{-1}$ is a proper ideal in $\gg^M_0$.  

Assume first that $I_1:=I \cap \gg_1^M$ is not contained in 
$\oo(V)$  (respectively, $\sp(V$) and let $X\in I_1\setminus \oo(V) $ (respectively, $X\in I_1\setminus \sp(V) $). 
By an argument similar to the one at the end of the proof of Corollary \ref{cor_2}, there exists $Y\in \gg^M_{-1}$ such that 
$[Y,X]\notin  \gl(V,V_*)\oplus \CC D$. Therefore by Corollary \ref{cor_2} we obtain a contradiction with our assumption that $I_0$ is a proper ideal in $\gg^M_0$.

Thus, we have proved that $I_1\subset \oo(V)$   (respectively, $\sp(V)$) and, similarly, $I_{-1}:=I \cap \gg_{-1}^M \subset \oo(V)$   (respectively, $\sp(V)$). 
Moreover, $I_0\subset  \gl (V,V_*)$ by Corollary  \ref{cor_2}. But then $I$ is a non-zero ideal in  $\oo(V)$   (respectively, $\sp(V)$).
Since both  $\oo(V)$ and $\sp(V)$ are simple, the statement follows.
\end{proof}

The proof of Theorem \ref{maintheo} is complete.
\end{proof}

\vspace{0.33cm}

Theorem \ref{maintheo} a) gives a complete list of ideals in  $\gl^M(V,V_*)$ for a countable-dimensional $V$. 
Indeed, since $\sl(V,V_*)$ is a simple Lie algebra, we obtain that all 
proper non-zero ideals in $\gl^M(V,V_*)$ are $\gl(V,V_*)$, $\sl(V,V_*)$, $\CC\id$,   $\sl(V,V_*)\oplus\CC\id$ and $\gl(V,V_*)\oplus\CC\id$.
In the same way the Lie algebra $\End(V)$ also has five proper non-zero ideals. 

Note that if $V$ is not countable-dimensional, then $\gl^M(V,V_*)$, $\End(V)$ and  $\oo^M(V)$ (respectively, $\sp^M(V)$)  have the following ideal:
\begin{equation}\label{ideal}
\lbrace X \;|\; \mathrm{dim}\,(X\cdot V) \mathrm{\;is\;finite\;or\;countable}\rbrace.
\end{equation}
Hence, Theorem \ref{maintheo} does not hold in this case.

\vspace{0.33cm}

\section{Dense subalgebras}\label{sect:7}
\subsection {Definition and general results}\label{sect:7.1}
\begin{defi}\label{dense_salg}Let $\ll$ be a Lie algebra, $R$ be an $\ll-$module, $\kk\subset \ll$ be a Lie subalgebra. We say that $\kk$ acts \textit{densely}
on $R$ if for any finite set of vectors $r_1,\dots,r_n \in R$ and any $l\in \ll$ there is $k\in\kk$ such that $k\cdot r_i = l\cdot r_i$ for $i=1,\dots,n$.
\end{defi}

\begin{lemma}\label{denshom}  Let $\kk\subset\ll$ and let $R,N$ be two $\ll$-modules such that $\kk$ acts densely on $R\oplus N$. Then 
$\Hom_\ll(R,N)=\Hom_\kk(R,N)$.
\end{lemma}

\begin{proof} There is an obvious inclusion  $\Hom_\ll(R,N)\subset\Hom_\kk(R,N)$. Suppose there exists $\varphi\in \Hom_\kk(R,N)\setminus \Hom_\ll(R,N)$.
Then one can find $r\in R$, $l\in\ll$ such that $\varphi(l\cdot r)\neq l\cdot \phi(r)$. Since $\kk$ acts densely on $R \oplus N$,
there exists $k\in\kk$ such that $k \cdot r = l \cdot r$ and $k \cdot \varphi(r) = l \cdot \varphi (r)$.
Therefore we have
$$\phi(l \cdot r)=\varphi(k \cdot r)= k \cdot \phi(r) = l \cdot \phi(r).$$
Contradiction.
\end{proof}

\begin{lemma}\label{densprop} Let $\kk\subset\ll$ and $R$ be an $\ll$-module on which $\kk$ acts densely. Then 

a) $\kk$ acts densely on any $\ll-$subquotient of $R$;

b) $\kk$ acts densely on $R^{\otimes n}$ for $n \geq 1$;

c)  $\kk$ acts densely on $R^{\oplus n}$ for $n \geq 1$;

d)  $\kk$ acts densely on $T(R)^{\oplus n}$ for $n \geq 1$.
\end{lemma}
\begin{proof} a) Let $N$ be an $\ll$-submodule of $R$. It follows immediately from the definition that $\kk$ acts densely on $N$
and on $R/N$. That implies the statement.

b) Let $r_1,\dots, r_q\in R^{\otimes n}$. Write 
$$r_i=\sum_{j=1}^{s(i)}m^i_{j1}\otimes\dots\otimes m^i_{jn}$$
for some $m^i_{jp}\in R$. For any $l\in \ll$ there exists $k\in\kk$ such that $k \cdot m^i_{jp} = l \cdot m^i_{jp}$ for all $i\leq r$, $p\leq n$ and $j\leq s(i)$.
Then $k \cdot r_i = l \cdot r_i$ for all $i\leq q$.

Proving c) and d) is similar to proving b) and we leave it to the reader.
\end{proof}

\begin{lemma}\label{densprop1} Let $\kk, \ll$ and $R$ be as in Lemma \ref{densprop}. Then a $\kk$-submodule of $R$ is $\ll$-stable. Hence any $\kk$-subquotient of $R$ has a natural structure of $\ll$-module.
\end{lemma}
\begin{proof} Straightforward from the definition.
\end{proof}

\begin{theo}\label{theodense} Let $\mathcal C_\ll$ be a full abelian subcategory of $\ll-\operatorname{mod}$ such that $\kk$ acts densely on any object in $\mathcal C$.
Let $\operatorname {Res}:\ll-\operatorname{mod}\to \kk-\operatorname{mod}$ be the functor of restriction. Let $\mathcal C_\kk$ be the image of  $\mathcal C_\ll$ 
under  $\operatorname {Res}$. Then  $\mathcal C_\kk$ is a  full abelian subcategory of $\kk-\operatorname{mod}$ and $\operatorname{Res}$ induces an equivalence of $\mathcal C_{\kk}$ and $\mathcal C_{\ll}$.
\end{theo}

\begin{proof}
The first assertion follows from Lemma \ref{denshom}. It also follows from the same lemma that  $\operatorname{Res}(R)\simeq \operatorname{Res}(N)$
implies $R\simeq N$. Thus, every object in  $\mathcal C_\kk$ has a unique (up to isomorphism) structure of $\ll$-module. This provides a quasi-inverse of 
$\operatorname{Res}$. Hence the second assertion holds.
\end{proof}

Let $R$ be an $\ll$-module.
Denote by $\TT^R_{\ll}$ the full subcategory of $\ll$-mod consisting of all finite length subquotients of finite direct sums $T(R)^{\oplus n}$ for $n \geq 1$.

\begin{prop}\label{prop_eqcat} Let $\kk,\ll$ and $R$ be as in Lemma 8.2. Then the restriction functor
\begin{equation}\label{rest_fact}
 \operatorname{Res}:\TT^R_{\ll} \rightsquigarrow \TT^R_{\kk}
\end{equation}
is an equivalence of monoidal categories.
\end{prop}

\begin{proof} By Lemma \ref{densprop},  $\operatorname{Res}(\TT^R_{\ll})=\TT^R_{\kk}$. Thus $\operatorname{Res}$ is an equivalence of $\TT_{\ll}^R$ and $\TT_{\kk}^R$ by Theorem \ref{theodense}. In addition, $\operatorname{Res}$ clearly commutes with $\otimes$, hence the statement.
\end{proof}

\subsection{Dense subalgebras of Mackey Lie algebras}\label{sect:7.2}

Let now $\gg^M$ denote one of the Lie algebras $\gl^M(V,W), \oo^M(V), \sp^M(V)$, and $\gg$ denote respectively the subalgebra $\gl(V,W), \oo(V), \sp(V)$.
By $R$ we denote the $\gg^M$-module $V\oplus W$ (respectively, $V$).

In what follows we call a Lie subalgebra 
$\aa\subset \gg^M$ \textit{dense} if it acts densely on $R$.
It is easy to see that $\gg$ is a dense subalgebra of $\gg^M$.

Here are further examples of dense subalgebras of $\gl^M(V,V_*)$ for a countable-dimensional space $V$. We identify $\gl^M(V,V_*)$ with the Lie algebra of matrices $(x_{ij})_{i\geq1,j\geq1}$ each row and column of which are finite. 

1. The Lie algebra $\jj(V,V_*)$ consisting of matrices $J = (x_{ij})_{i\geq1,j\geq1}$ such that $x_{ij}=0$ when 
$|i-j|>m_J$ for some $m_J \in \ZZ_{>0}$ (generalized Jacobi matrices), is dense in $\gl^M(V,V_*)$.

2. The subalgebra $\mathfrak{lj}(V,V_*) \subset \jj(V,V_*)$ consisting of matrices $X=(x_{ij})_{i\geq1,j\geq1}$ satisfying the condition $x_{ij}=0$ when $i-j>c_Xj$ for some $c_X \in \ZZ_{>0}$, is dense in $\gl^M(V,V_*)$.

3.  The subalgebra $\mathfrak{pj}(V,V_*)$ of matrices $Y=(x_{ij})_{i\geq1,j\geq1}$ satisfying the condition $x_{ij}=0$ when 
$i-j>p_Y(j)$ for some polynomial $p_Y(t) \in \ZZ_{\geq0}[t]$, is dense in $\gl^M(V,V_*)$.

4. Let $\gg$ be a countable-dimensional diagonal Lie algebra. If $\gg$ is of type $\sl$,  
fix a chain (\ref{def_gg}) of diagonal embeddings where $\gg_i \simeq \sl(n_i)$. Observe that given a chain (\ref{inf_ch_nat}), we can always choose a chain
$$V^*_{\oj_{1}} \stackrel{\mu_1}{\hookrightarrow} V^*_{\oj_{2}} \stackrel{\mu_2}{\hookrightarrow}  \dots \hookrightarrow V^*_{\oj_{i}} \stackrel{\mu_i}{\hookrightarrow} V^*_{\oj_{i+1}} \hookrightarrow \dots$$ so that the non-degenerate pairing 
$V_{\oj_i} \times V^*_{\oj_i} \rightarrow \CC$ restricts to a non-degenerate pairing $\ae_i(V_{\oj_i}) \times \mu_i(V^*_{\oj_i}) \rightarrow \CC$. 
Therefore, by multiplying $\mu_i$ by a suitable constant, we can assume that $\ae_i$ and $\mu_i$ preserve the natural pairings 
$V_{\oj_i} \times V^*_{\oj_i} \rightarrow \CC$. 
This shows that, given a natural representaion $V$ of $\gg$, there always is a natural representation $V_*$ such that there is a non-degenerate $\oj$-invariant pairing $V \times V_* \rightarrow \CC$. 
This gives an embedding of $\gg$ as a dense subalgebra in $\gl^M(V,V_*)$

If $\gg$ is of type $\oo$ or $\sp$, then a natural representation $V$ of $\gg$ is defined again by a chain of embeddings (\ref{inf_ch_nat}). Moreover, $V$ always carries a respective non-degenerate symmetric or symplectic form. Therefore $\gg$ can be embedded as a dense subalgebra in $\oo^M(V)$, or respectively in $\sp^M(V)$. 

The following statement is a particular case of Proposition \ref{prop_eqcat}. 

\begin{corollary}\label{prop_squot} Let $\aa$ be a dense subalgebra in $\gg^M$. Then the monoidal categories $\TT^R_{\gg^M}$ and $\TT^R_{\aa}$ are equivalent.
\end{corollary}

\subsection {Finite corank subalgebras of $\gg^M$ and the category $\TT_{\gg^M}$}\label{sect:7.3} 
We now generalize the notion of finite corank subalgebra to Mackey Lie algebras.

Let $V_f\subset V, W_f\subset W$ be a non-degenerate pair of finite-dimensional subspaces. Then $\gl(W_f^\perp,V_f^\perp)$ is a subalgebra of $\gl^M(W_f^\perp,V_f^\perp)$ and also a subalgebra of   
$\gl^M(V,W)$. Moreover, the following important relation holds
\begin{equation}\label{gl}
\sl(V,W)/\sl(W_f^\perp,V_f^\perp)=\gl(V_f,W_f)\oplus (V_f\otimes V_f^\perp) \oplus (W_f^\perp\otimes W_f) = \gl^M(V,W)/\gl^M(W_f^\perp,V_f^\perp).
\end{equation} 

We call a subalgebra $\kk\subset \gl^M(V,W)$ a \textit{finite corank subalgebra} if it contains  
$\gl^M(W_f^\perp,V_f^\perp)$ for some non-degenerate pair $V_f \subset V, W_f \subset W$.

Similarly, let $V$ be a vector space equipped with a symmetric (respectively, skew-symmetric) non-degenerate form and $V_f$ be a non-degenerate finite-dimensional
subspace. We have a well-defined subalgebra $\oo^M(V_f^\perp)\subset \oo^M(V)$
(respectively, $\sp^M(V_f^\perp)\subset \sp^M(V)$). Furthermore,

\begin{equation}\label{0}
\oo(V)/\oo(V_f^\perp)=\oo(V_f)\oplus (V_f\otimes V_f^\perp)=\oo^M(V)/\oo^M(V_f^\perp),\,\,
\sp(V)/\sp(V_f^\perp)=\sp(V_f)\oplus (V_f\otimes V_f^\perp)=\sp^M(V)/\sp^M(V_f^\perp).
\end{equation} 

We call $\kk\subset \oo^M(V)$ (respectively, $\sp^M(V)$) a \textit{finite corank subalgebra}  if  
it contains $ \oo^M(V_f^\perp)$ (respectively, $\sp^M(V_f^\perp)$) for some $V_f$ as above.

Next, we say that $\gg^M$-module $L$ {\it satisfies the large annihilator condition} if the annihilator in $\gg^M$ of any $l\in L$
contains a finite corank subalgebra. It follows immediately from the 	definition that if $L_1$ and $L_2$ satisfy the large annihilator condition,
then the same is true for $L_1\oplus L_2$ and $L_1\otimes L_2$.

\begin{lemma}\label{corank2} 
Let $L$ be a $\gg^M$-module which is integrable as a $\gg-$module. 
If $L$ satisfies the large annihilator condition (as a $\gg^M-$module), then $\gg$ acts densely on $L$. 
\end{lemma}
\begin{proof} Since $L$ satisfies the large annihilator condition as a $\gg^M-$module, so does also $L^{\oplus n}$. It suffices to show that for all $n \in \ZZ_{\geq 1}$ and all $l\in L^{\oplus n}$ 
we have 
\begin{equation}\label{eq:sl_u}
 \gg\cdot l=\gg^M\cdot l.
\end{equation}
However, as $l$ is annihilated by $\gl^M(W_f^\perp,V_f^\perp)$ for an appropriate finite-dimensional non-degenerate pair $V_f \subset V, W_f \subset W$ in the case $\gg = \sl(V,W)$ (respectively, by $\oo^M(V_f^{\perp}),\sp^M(V_f^{\perp})$ in the case $\gg = \oo(V),\sp(V)$), (\ref{eq:sl_u}) follows from (\ref{gl}), (respectively, from (\ref{0})).
\end{proof}

\begin{lemma}\label{corank1} Let $L$ be a $\gg$-module satisfying the large annihilator condition. Then the $\gg$-module
structure on $L$ extends in a unique way to a $\gg^M$-module structure such that $L$ satisfies  the large annihilator condition as a $\gg^M$-module.
\end{lemma}

\begin{proof} Consider the case $\gg=\sl(V,W)$. Any $l\in L$ is annihilated by $\sl(W_f^\perp,V_f^\perp)$ for an appropriate finite-dimensional non-degenerate pair $V_f \subset V, W_f \subset W$.
Let $x\in \gl^M(V,W)$.
By (\ref{gl}) there exists $y\in \sl(V,W)$ such that $x+\gl^M(W_f^\perp,V_f^\perp)=y+\sl(W_f^\perp,V_f^\perp)$. Moreover, $y$ is unique modulo $\sl(W_f^\perp,V_f^\perp)$. Thus we can set $x \cdot l:=y \cdot l$. It is an easy check that this yields a well-defined $\gl^M(V,W)$-module structure on $L$ compatible with the $\sl(V,W)-$module structure on $L$. 

For $\gg = \oo(V),\sp(V)$ one uses (\ref{0}) instead of (\ref{gl}).
\end{proof}


\vspace{0.33cm}

We can now define the category $\TT_{\gg^M}$ as an analogue of the category $\TT_\gg$. More precisely, the category $\TT_{\gg^M}$ is the full subcategory of $\gg^M$-mod consisting of all modules of finite length, 
integrable over $\gg$ and satisfying the large annihilator condition.

The following is our main result in section \ref{sect:7}.
\begin{theo}\label{mainres_sect7}
 a) $\TT_{\gg^M} = \TT_{\gg^M}^R$, where $R = V\oplus W$ for $\gg = \sl(V,W)$ and $R=V$ for $\gg = \oo(V), \sp(V)$.
 
 b) The functor $\operatorname{Res}:\TT_{\gg^M} \rightsquigarrow \TT_\gg$ is an equivalence of monoidal categories.
\end{theo}
\begin{proof}
 It is clear that $\TT_{\gg^M}^R$ is a full subcategory of $\TT_{\gg^M}$. We need to show only that any $L \in \TT_{\gg^M}$ is isomorphic to a subquotient of $T(R)^{\oplus n}$ for some $n$. Obviously, $L$ satisfies the large annihilator condition as a $\gg-$module. Furthermore, by Lemma \ref{corank1} a), $\gg$ acts densely on $L$, hence $L$ has finite length as a $\gg-$module. By Corollary \ref{corr:1} b), $L$ is isomorphic to a $\gg-$subquotient of $T(R)^{\oplus n}$ for some $n$, and by Proposition \ref{prop_eqcat} $L$ is the restriction to $\gg$ of some $\gg^M-$subquotient $L'$ of $T(R)^{\oplus n}$. However, since $L'$ satisfies the large annihilator condition, Lemma \ref{corank2} implies that there is an isomorphism of $\gg^M-$modules $L\simeq L'$. This proves a).
 
 b) follows from a) and Proposition \ref{prop_eqcat}.
\end{proof}

\vspace{0.33cm}

The following diagram summarizes the equivalences of monoidal categories established in this paper:
$$ \TT_{\aa} \stackrel{\operatorname{Res}}{\leftsquigarrow} \TT_{\gg^M} = \TT_{\gg^M}^R \stackrel {\operatorname{Res}}{\rightsquigarrow} \TT_{\gg} \stackrel{\Phi}{\rightsquigarrow} \TT_{\gg_c}.$$
Here $\aa$ is any dense subalgebra of $\gg^M$ and $R = V\oplus W$ for $\gg = \sl(V,W)$, $R=V$ for $\gg = \oo(V), \sp(V)$. In particular, when $\gg = \sl(V,V_*)$ for countable-dimensional $V$ and $V_*$, $\aa$ can be chosen as the Lie algebra $\jj(V,V_*)$ or as any countable-dimensional diagonal Lie algebra.

\section{Further results and open problems}\label{sect:8}

Theorem \ref{mainres_sect7} a) can be considered an analogue of Theorem \ref{th:6.1} and Corollary \ref{corr:1} b) as it provides two equivalent descriptions of the category $\TT_{\gg^M}$. It is interesting to have a longer list of such equivalent descriptions.

The following proposition provides another equivalent condition characterizing the objects of $\TT_{\gg^M}$ under the additional assumption that 
$\gg= \sl(V,V_*), \oo(V), \sp(V)$ is countable dimensional.

\begin{prop}\label{prop_sect8}
Let $\gg^M=\gl^M(V,V_*), \oo^M(V)$, $\sp^M(V)$ for a countable dimensional $V$, and let $L$ be a $\gg^M-$module of finite length which is integrable as a $\gg-$module. Then $L$ is an object of $\TT_{\gg^M}$ if and only if $\gg$ acts densely on $L$.
\end{prop}

We first need a lemma.

\begin{lemma}\label{extension} Let $\gg^M=\gl^M(V,V_*), \oo^M(V)$, $\sp^M(V)$ for a countable dimensional $V$, and let $L$ and $L'$ be  
$\gg^M$-modules. 
Assume that $L$ and $L'$ have finite length as $\gg$-modules. Then 
$$\Hom_{\gg}(L,L')=\Hom_{\gg^M}(L,L').$$
In particular, if $L$ and $L'$ are isomorphic as $\gg$-modules, then $L$ and $L'$ are isomorphic as $\gg^M$-modules. 
\end{lemma}

\begin{proof} Observe that $\Hom_{\mathbb C}(L,L')$ has a natural structure of $\gg^M$-module defined by
\begin{equation}\label{eq_12_11}
(X\cdot \varphi)(l):=X\cdot\varphi(l)-\varphi (X\cdot l)\,\,\text{for}\,\, X\in\gg^M, \varphi\in \Hom_{\mathbb C}(L,L'),\,l\in L. 
\end{equation}
Since $\gg$ is an ideal in $\gg^M$, $\Hom_{\gg}(L,L')$ is a $\gg^M$-submodule in $\Hom_{\mathbb C}(L,L')$. Moreover, $\Hom_{\gg}(L,L')$ is finite dimensional as $L$ and $L'$ have finite length over $\gg$. On the other hand, Theorem \ref{maintheo} implies that $\gg^M$ does not have proper ideals of finite codimension, hence any finite-dimensional $\gg^M$-module is trivial. Therefore (\ref{eq_12_11}) defines a trivial $\gg^M-$module structure of $\Hom_{\gg^M}(L,L')$, which means that any $\phi \in \Hom_\gg(L,L')$ belongs to $\Hom_{\gg^M}(L,L')$. This shows that $\Hom_{\gg}(L,L')=\Hom_{\gg^M}(L,L')$.
The second assertion follows immediately.
\end{proof}

\vspace{0.3cm}

\begin{proof}\textit{of Proposition \ref{prop_sect8}}\hspace{0.2cm}If $L \in \TT_{\gg^M}$, then $\gg$ acts densely on $L$ by Lemma \ref{corank1}. 

Let now $\gg$ act densely on $L$. We first prove that $L$ satisfies the large annihilator condition as a $\gg-$module. 
Assume that $\gg$ acts densely on $L$ but $L$ does not satisfy the large annihilator condition as a $\gg-$module. Using the matrix realizations of $\gg$ and $\gg^M$
one can show that there exists $l\in L$ and a sequence $\lbrace X_i\rbrace_{i\in\ZZ_{\geq 1}}$ of commuting linearly independent elements  $X_i\in\gg$
which don't belong to the annihilator of $l$. Furthemore, 
one can find an infinite subseqence $\lbrace Y_j=X_{i_j} \rbrace$ such that each $Y_j$ lies in an $\sl(2)$-subalgebra
$\text{\ss}_j\subset\gg$ with the condition
$[\text{\ss}_j,\text{\ss}_s]=0$ for $j\neq s$. Then $\prod_j \text{\ss}_j$ is a Lie subalgebra in $\gg^M$, and let $\text{\ss}$ be the diagonal subalgebra in $\prod_j\,\text{\ss}_j$. 
If $x\in\text{\ss}$, we denote by $x_j$ its component in $\text{\ss}_j$.

Since $\gg$ acts densely on $L$, there exists a linear map $\theta:\text{\ss}\to \gg$ such that $\theta(y)\cdot l=y\cdot l$ for all $y\in\text{\ss}$. 
On the other hand, there exists $n \in \ZZ_{\geq 1}$ such that
$[\theta(y),x_j]=0$ for all $y,x\in\text{\ss}$ and $j>n$. Let $d_y:=y-\theta(y)$. Then $d_y\cdot l=0$ and 
\begin{equation}\label{equality_d_l}
[d_y,x_j]=[y,x_j]=[y_j,x_j]\;\;\;\mathrm{for\;all}\;x,y\in\text{\ss}\;\mathrm{and}\;j>n.
\end{equation}           
Set $L_j:=U(\text{\ss}_j)\cdot l$. Then (\ref{equality_d_l}) implies $d_y\cdot L_j \subset L_j$ for all $j>n$.  
Moreover, $\psi_y:=d_y-y_j$ commutes with $\text{\ss}_j$, hence $\psi_y\in\End_{\text{\ss}_j}(L_j)$. 
Considering $y_j + \psi_y$ as an element of $\End_{\CC}(L_j)$, we obtain in addition that $l\in\Ker (y_j+\psi_y)$ for all $y\in\text{\ss}$ and all $j>n$. 

Choose a standard basis $E,H,F\in\text{\ss}$. Since $L_j$ is a finite-dimensional $\beta_j\simeq\sl(2)$-module, 
we obtain easily
$$\Ker(E_j+\psi_E)\cap L_j=L_j^{E_j},\,\, \Ker(F_j+\psi_F)\cap L_j=L_j^{F_j}.$$
Since
$$l \in \Ker(E_j+\psi_E)\cap \Ker(F_j+\psi_F)\cap L_j=L_j^{\text{\ss}_j},$$ 
we conclude that $L_j$ is a trivial  $\text{\ss}_j$-module for all $j>n$, which contradicts our original assumption that $Y_j\cdot l\neq 0$. 
Thus, $L$ satisfies the large annihilator condition as a $\gg-$module.

Note that, as $\gg$ acts densely on $L$, the length of $L$ as a $\gg-$module is the same as the length of $L$ as a $\gg^M-$module. 
Since $L$ satisfies  the large annihilator condition for $\gg$ and has finite length as $\gg-$module, we conclude that $L_{\downarrow\gg}$ is a tensor module, i.e. 
an object of $\TT_\gg$. By Theorem \ref{mainres_sect7} b), $L_{\downarrow\gg} = L_{\downarrow\gg}'$ for some $L' \in \TT_{\gg^M}$. Finally, Lemma \ref{extension} implies that the $\gg^M$ modules $L'$ and $L$ are isomorphic, i.e. $L\in\TT_{\gg^M}$.
\end{proof}
\vspace{0.33cm}

Next, under the assumption that $V$ is countable dimensional, consider maximal subalgebras $\hh^M$ of $\gg^M$ which act semisimply on $V$ and $V_*$ (respectively only on $V$ for $\gg = \oo(V), \sp(V)$). 
It is strightforward to show that the centralizer in $\gg^M$ of any local Cartan subalgebra $\hh$ of $\gg$ is such a subalgebra of $\gg^M$. 
If $\gg^M = \gl^M(V,V_*)$ is realized as the Lie algebra of matrices $X = (x_{ij})_{i,j \in \ZZ}$ with finite rows and columns, then $\hh^M$ can 
be chosen as the subalgebra of diagonal matrices.

The following statement looks plausible to us.

\begin{conjecture}\label{conject:1}
Let $\gg = \sl(V,V_*), \oo(V), \sp(V)$ for a countable-dimensional $V$. Let $M$ be a finite length $\gg^M$-module which is integrable as a $\gg-$module. The following conditions on $M$ are equivalent:
\begin{enumerate}[a)]
 \item $M \in \TT_{\gg^M}$;
 \item $M$ is countable dimensional;
 \item $M$ is a semisimple $\hh^M-$module for some subalgebra $\hh^M\subset \gg^M$;
  \item $M$ is a semisimple $\hh^M-$module for any subalgebra $\hh^M \subset \gg^M$.
\end{enumerate}

\end{conjecture} 

\vspace{0.4cm}

Consider now the inclusion of Lie algebras 
$$\gg = \sl(V,V_*) \subset \gl^M(V,V^*) = \End(V)$$
where $V$ is an arbitrary vector space. The subalgebra $\gg$ is not dense in $\End(V)$, nevertheless the monoidal categories $\TT_{\gg}$ and $\TT_{\End(V)}$ are equivalent by Theorem \ref{th:6.1} and Theorem \ref{mainres_sect7}. Here is a functor which most likely also provides such an equivalence. Let $M \in \TT_{\End(V)}$. Set 
$$\Gamma_{\gg}^\wt \;(M) := \cap_{\hh \subset \gg}\,\Gamma_{\hh}^\wt\,(M)$$
where $\hh$ runs over all local Cartan subalgebras of $\gg$.

\begin{conjecture}
 $\Gamma_{\gg}^\wt: \TT_{\End(V)} \rightsquigarrow \TT_{\gg}$ is an equivalence of monoidal categories.
\end{conjecture}

If $V$ is countable dimensional, it is easy to check that $V^*/V_*$ is a simple $\gg^M = \gl^M(V,V_*)-$module. Hence $V^*$ is a $\gg-$module of length 2. This raises the natural question of whether the entire category $\TT_{\End(V)}$ consists of $\gg-$modules of finite length. A further problem is to compute the socle filtration as a $\gg-$module of a simple $\End(V)-$module in $\TT_{\End(V)}$.

Another open question is whether there is an analogue of the category $\widetilde{\mathrm{Tens}}_{\gg}$ when we replace $\gg$ by $\gg^M$. More precisely, 
what can be said about the abelian monoidal category of $\gg^M-$modules obtained from $\TT_{\gg^M}$ by iterated dualization in addition to taking submodules, 
quotients and applying $\otimes$ ? In particular, the adjoint representation, and therefore the coadjoint representation are objects of 
$\widetilde{\mathrm{Tens}}_{\gg^M}$.  How can one describe the coadjoint representation $(\gg^M)^*$ of $\gg^M$?

\end{document}